\documentclass[sn-mathphys,Numbered]{sn-jnl}

\usepackage{graphicx}%
\usepackage{multirow}%
\usepackage{amsmath,amssymb,amsfonts}%
\usepackage{amsthm}%
\usepackage{mathrsfs}%
\usepackage[title]{appendix}%
\usepackage{xcolor}%
\usepackage{textcomp}%
\usepackage{manyfoot}%
\usepackage{booktabs}%
\usepackage{algorithm}%
\usepackage{algorithmic}%
\usepackage{listings}%
\usepackage{nicefrac}       
\usepackage{microtype}      
\usepackage{subfigure}
\usepackage[numbers]{natbib}
\usepackage{doi}
\usepackage{tikz}
\usepackage{enumerate}
\usepackage{threeparttable}
\usepackage{longtable}
\usepackage{hyperref}
\usepackage{cleveref}
\usepackage{verbatim}
\usepackage{fullpage}

\newtheorem{thm}{Theorem}
\newtheorem{prop}{Proposition}%
\newtheorem{lem}{Lemma}
\newtheorem{coro}{Corollary}

\newtheorem{rmk}{Remark}%

\newtheorem{asmp}{Assumption}

\newcommand{\R}{\mathbb R}

\newcommand{\N}{\mathbb N}

\newcommand{\tilr}{\Tilde{r}}

\newcommand{\bgrad}{\textbf{grad}}
\newcommand{\Bt}{\textbf{t}}

\newcommand{\textand}{\quad \text{and} \quad}

\DeclareMathOperator{\grad}{grad}

\DeclareMathOperator*{\argmin}{arg\,min}
\newcommand{\norm}[1]{\left \Vert #1 \right \Vert}

\newcommand{\set}[1]{\left \{ #1 \right \}}

\newcommand{\dist}{\mathrm{dist}}
\newcommand{\TM}[1]{T_{#1}\mathcal{M}}
\newcommand{\tm}{T\mathcal{M}}
\newcommand{\rmd}{\mathrm{d}}
\newcommand{\Exp}{\operatorname{Exp}}
\newcommand{\inprod}[2]{\left\langle #1, #2\right\rangle}
\newcommand{\dx}[2]{\Delta x_{#1}^{(#2)}}
\newcommand{\dr}[2]{\Delta r_{#1}^{(#2)}}

\newcommand{\cB}{\mathcal{B}}

\newcommand{\cE}{\mathcal{E}}

\newcommand{\cH}{\mathcal{H}}

\newcommand{\cL}{\mathcal{L}}
\newcommand{\cM}{\mathcal{M}}
\newcommand{\cN}{\mathcal{N}}
\newcommand{\cO}{\mathcal{O}}
\newcommand{\cP}{\mathcal{P}}

\newcommand{\cR}{\mathcal{R}}
\newcommand{\cS}{\mathcal{S}}
\newcommand{\cT}{\mathcal{T}}
\newcommand{\cU}{\mathcal{U}}

\newcommand{\cX}{\mathcal{X}}

\begin{document}

\title{Riemannian Anderson Mixing Methods for Minimizing $C^2$-Functions on Riemannian Manifolds}


\author[1]{\fnm{Zanyu} \sur{Li}}\email{lizy21@mails.tsinghua.edu.cn}

\author*[1,2]{\fnm{Chenglong} \sur{Bao}}\email{clbao@tsinghua.edu.cn}

\affil[1]{\orgdiv{Yau Mathematical Sciences Center}, \orgname{Tsinghua University}, \orgaddress{\city{Beijing}, \country{China}}}

\affil[2]{\orgname{Yanqi Lake Beijing Institute of Mathematical Sciences and Applications}, \orgaddress{\city{Beijing}, \country{China}}}



\abstract{The Anderson Mixing (AM) method is a popular approach for accelerating fixed-point iterations by leveraging historical information from previous steps. In this paper, we introduce the Riemannian Anderson Mixing (RAM) method, an extension of AM to Riemannian manifolds, and analyze its local linear convergence under reasonable assumptions. Unlike other extrapolation-based algorithms on Riemannian manifolds, RAM does not require computing the inverse retraction or inverse exponential mapping and has a lower per-iteration cost. Furthermore, we propose a variant of RAM called Regularized RAM (RRAM), which establishes global convergence and exhibits similar local convergence properties as RAM. Our proof relies on careful error estimations based on the local geometry of Riemannian manifolds. Finally, we present experimental results on various manifold optimization problems that demonstrate the superior performance of our proposed methods over existing Riemannian gradient descent and LBFGS approaches.}

\keywords{manifold optimization, regularized Anderson mixing, local linear convergence, global convergence}


\maketitle

\section{Introduction}
Manifold optimization has recently attracted significant interest due to its diverse range of applications, including matrix completion problems \cite{vandereycken2013low}, nonlinear eigenvalue problems \cite{zhang2014gradient}, principal component analysis \cite{
zass2006nonnegative}, sparse blind deconvolution \cite{zhang2017global}, and deep learning \cite{cho2017riemannian}. The general problem of manifold optimization involves minimizing a real-valued function on a complete Riemannian manifold, as shown in Equation \eqref{eqn:model}:
\begin{equation}\label{eqn:model}
	\min_{x\in\mathcal{M}} f(x),
\end{equation}
where $f$ is a $C^2$ function defined on the Riemannian manifold $\mathcal{M}$. Various classic Euclidean optimization methods have been extended to the Riemannian setting since Luenberger's pioneering work \cite{luenberger1972gradient} on the gradient method over manifolds, with Iannazzo \cite{iannazzo2018riemannian} studying Barzilai-Borwein methods, Huang et al. \cite{huang2018riemannian} examining BFGS quasi-Newton methods, Absil \cite{absil2007trust} developing trust-region methods, and Sato et al. \cite{sato2021riemannian} analyzing conjugate gradient methods. Moreover, adaptive regularized Newton algorithms \cite{hu2018adaptive} and stochastic methods \cite{sato2019riemannian} have been developed for solving large-scale Riemannian optimization problems. In addition, nonsmooth manifold optimization has garnered attention, with various algorithms proposed and analyzed, including subgradient methods \cite{ferreira1998subgradient}, proximal gradient methods \cite{huang2021riemannian}, semi-smooth Newton methods \cite{zhou2021semi}, and alternating direction methods of multipliers (ADMM) \cite{kovnatsky2016madmm}.

Extrapolation methods are widely used in optimization and numerical computation to accelerate first order algorithms. Among them, a typical extrapolation scheme is the Nesterov acceleration method \cite{nesterov1983method} which has been generalized to the manifold~\cite{zhang2018estimate,kim2022nesterov}. In~\cite{zhang2018estimate}, the Riemannian Nesterov method has the following update rule:
\begin{equation}
    \label{update:Nesterov}
\left\{
\begin{aligned}
    &y_{k} = \Exp_{x_k}\left(\frac{\alpha \gamma}{\gamma+\alpha \mu} \Exp_{x_k}^{-1}\left(v_k\right)\right),  \\
      &x_{k+1} =\Exp_{y_k}\left(-h \grad f\left(y_k\right)\right), \\
      &v_{k+1} =\Exp_{y_k}\left(\frac{(1-\alpha) \gamma}{\bar{\gamma}} \Exp_{y_k}^{-1}\left(v_k\right)-\frac{\alpha}{\bar{\gamma}} \grad f\left(y_k\right)\right),
\end{aligned}
\right.  
\end{equation}
where $h$ is the step size and $\alpha,\beta,\gamma,\bar{\gamma},\mu$ are predefined constants.
However, due to the existence of inverse exponential map or inverse retraction map in \eqref{update:Nesterov}, the above method can be computationally expensive in each step. Moreover, its convergence results are based on the "convexity" of $f$ in a manifold, which limits its applications.

Another commonly employed extrapolation technique in iterative schemes is the Anderson mixing (AM) method, as described by Anderson in his seminal work \cite{anderson1965iterative}. This method leverages information from past iterations to accelerate convergence towards a fixed-point solution.  Let $g:\R^{n}\rightarrow\R^n$ be a mapping, the goal of AM is to find the fixed points of $g$, i.e.\ the point $x$ with $g(x) = x$. Given $m\in\mathbb{N}$, we define the residual as $r_k = g(x_k)-x_k$ and 
\begin{equation}\label{eqn:X-and-R}
    X_k=[\Delta x_{k-m},\Delta x_{k-m+1},\ldots,\Delta x_{k-1}], R_k = [\Delta r_{k-m},\Delta r_{k-m+1},\ldots,\Delta r_{k-1}],
\end{equation}
where $\Delta$ denotes the discrete forward difference, e.g.\ $\Delta x_k = x_{k+1}-x_k$.
Each iteration in AM($m$) consists of two steps 
\begin{equation}\label{eqn:am_original_iter}
    \left\{
\begin{aligned}
& \Gamma_k = {\operatorname{argmin}} \left\| r_k - R_k\Gamma \right\|_2,\\
& \Bar{x}_k = x_k-X_k\Gamma_k, \ \bar{r}_k = r_k - R_k\Gamma_k,\\
& x_{k+1} = \Bar{x}_k+\beta_k \Bar{r}_k,
\end{aligned}
\right.
\end{equation}
where $\beta_k > 0$ is a parameter. The one-step update form of AM($m$) is
\begin{equation}
    x_{k+1} = x_k + \beta_k r_k - (X_k+\beta_kR_k)\Gamma_k.
\end{equation}
Also, many variants of AM have been proposed that use different projection conditions for finding the extrapolation coefficients~$\Gamma_k$. Although AM($m$) has close relationships with classical methods such as the multisecant Broyden method~\cite{fang2009two} and GMRES method~\cite{fang2009two,potra2013characterization}, it differs in that it does not explicitly approximate the Jacobian matrix or the product of the Jacobian matrix and a vector. This enables fast computation, particularly for large-scale optimization problems.Recent research has successfully applied AM($m$) to a range of problems, including correlation matrix completion, seismic data inversion, electronic structure computation, and deep learning~\cite{wei2021class}. 

From the theoretical perspective, the linear convergence of Anderson mixing (AM) is proven in \cite{toth2015convergence} for a contractive $g$ and uniformly bounded $\Gamma_k$, but the method fails to achieve faster convergence than classical fixed-point iteration. Evans et al. \cite{evans2020proof} extend AM to Hilbert space and obtain an improved local linear convergence rate under second-order assumptions on $g$, leveraging Taylor expansions. Chupin et al. \cite{chupin2021convergence} derive a superlinear convergence result for an adaptive-depth AM by assuming $\|F\circ g(x)\|_2 \leq K \|F(x)\|_2$ with $K \in (0,1)$, which has similarities to a contraction. Lastly, Pollock et al. \cite{pollock2021anderson} explore AM for both contractive and noncontractive operators, providing an improved local linear convergence rate that does not require $g$ to be contractive.

Motivated by the aforementioned computational and theoretical advantages of AM method, we propose the Rimenanian Anderson Mixing (RAM) method and aim to solve optimization problems on Riemannian manifolds, where the first-order condition $\grad f(x) = 0$ must be satisfied, with $\grad$ denoting the gradient on the manifold. Instead, we consider the problem of finding $x\in\cM$ such that $F(x) = 0$, where $F:\cM\rightarrow \tm$ is a $C^1$-vector field. To mimic the AM method, we define a fixed-point mapping $g:\cM\rightarrow\cM$ associated with the vector field $F$ as $g(x) := \Exp_{x}(-F(x))$, and it can be seen that $F(x) = 0$ is equivalent to $g(x)=x$. However, the forward difference operator $\Delta$ in~\eqref{eqn:X-and-R} is not well-defined on Riemannian manifolds, so we redefine the historical information matrices $X_k$ and $R_k$ by applying vector transport on each column in $X_k$ and $R_k$ so that both matrices lie in the tangent space of the same point. After solving a least-square problem in the tangent space, we compute the update using the retraction, which is an approximation of the exponential mapping. RAM does not require computing the inverse of retraction or exponential mapping, which improves computational efficiency. We establish the local linear convergence analysis of RAM under some assumptions on manifolds and the contractive assumption of $g$, which provides an improved local linear convergence rate over fixed-point iteration. As classical AM may cause problematic behavior, we propose a variant named regularized RAM (RRAM) that adds adaptive regularization to the least square problem. We prove that RRAM has global convergence properties while maintaining the local convergence property of RAM, due to careful estimation of iterations on the manifold. We conduct extensive experiments and compare RAM and RRAM methods with Riemannian gradient descent and LBFGS methods, demonstrating the numerical advantages of our proposed methods.

The rest of the paper is organized as follows. Section \ref{sec:pre} presents preliminaries on manifolds and some basic assumptions and lemmas. We present the details of our RAM algorithm and prove the local linear convergence in section \ref{sec:alg}. In section \ref{sec:global_conv}, we add some global convergence safeguards to RAM and propose the RRAM algorithm with a global convergence property. We report numerical experiments of several manifold optimization problems in section \ref{sec:experiment}. Finally, we conclude this paper in section \ref{sec:conclusion}.

\section{Notations and Preliminaries}\label{sec:pre}
In this section, we present some basic notations and definitions regarding Riemannian manifolds. For more details, interested readers may refer to \cite{absil2009optimization}.

Let $\inprod{\cdot}{\cdot}$ be the inner product on $\cM$, and let $\|\cdot\|$ be the induced metric, while we denote by $\|\cdot\|_2$ and $\|\cdot\|_\infty$ the $\ell_2$-norm and the $\ell_\infty$-norm of an Euclidean vector, respectively. The Frobenius norm of a matrix is denoted by $\|\cdot\|_F$. We define $B_{\cM}(p,r)$ as the open ball $\{q \in \mathcal{M}: \dist(p, q)<r\}$, and $B_{\TM{p}}(v_p,r)$ as the open ball $\{u_p \in\TM{p}: \|u_p-v_p\|<r\}$. Moreover, given $x\in\cM$, we denote by $\TM{x}$ the tangent space at $x$, and by $\tm$ the tangent bundle.

The distance between two points on the manifold $\cM$ is denoted by $\dist(x,y)$, which is defined as the infimum of the length of the piecewise smooth curves that connect $x$ and $y$ on $\cM$. We define $\cX(\cM)$ as the set of vector fields on $\cM$, and for all $X,Y\in\cX(\cM)$, we define $\nabla_X Y$ as the Levi-Civita connection of $X$ and $Y$. For a fixed vector field $X\in\cX(\cM)$, we say $H:\tm\mapsto\tm$ is the Jacobian of $X$ if $H(x):\TM{x}\rightarrow\TM{x}$ is a linear map satisfying $H(x)[v_x] = \nabla_{v_x} X(x)$ for all $v_x\in\TM{x}$. 

For a given smooth curve $\gamma: [0,1]\rightarrow \cM$ and a vector field $X$, we say that $X$ is parallel along $\gamma$ if $\nabla \dot{\gamma}X = 0$, where $\dot{\gamma}(t)\in T_{\gamma(t)}\cM$. Specifically, we call $\gamma$ a geodesic on $\cM$ if $\nabla_{\dot{\gamma}}\dot{\gamma}=0$. Assuming a geodesic $\gamma$ satisfies $\gamma(0)=x$, $\dot\gamma(0)=\eta$, and $\gamma(1)$ exists, we define the exponential map at $x$ for $\eta\in\TM{x}$ as $\Exp_x(\eta)=\gamma(1)$.

We define a retraction $\cR:\tm\rightarrow\cM$ as a $C^2$-mapping that satisfies (1) $\cR(0_x) = x$ for all $x\in\cM$, and (2) $\frac{\rmd}{\rmd t}\cR(t\xi_x)|_{t=0}=\xi_x$ for all $\xi_x\in\TM{x}$. We denote by $\cR_x$ the restriction of $\cR$ to $\TM{x}$, and by $\cT:\tm\oplus\tm\rightarrow \tm$ the vector transport associated with retraction $\cR$. Specifically, $\cT$ is a $C^2$-mapping that satisfies (1) $\cT_{\eta_x}\xi_x\in\TM{\cR_x(\eta_x)}$, and (2) $\cT_{\eta_x}(\cdot)$ is linear. If $\cR^{-1}_x y$ exists for $x,y\in\cM$, we define $\cT_{x}^{y}$ as $\cT_{\cR^{-1}_x y}(\cdot)$. A vector transport $\cT$ is called isometric if $\inprod{\cT_{\eta_x}\zeta_x}{\cT_{\eta_x}\xi_x} = \inprod{\zeta_x}{\xi_x}$ for all $\zeta_x,\xi_x\in\TM{x}$. Given any retraction $\cR$, we can define a vector transport, denoted by $\cT_{\cR}$, by $\cT_{\cR_{\eta_x}}\xi_x := \frac{\rmd}{\rmd t}\cR_x(\eta_x+t\xi_x)|_{t = 0} $. An important retraction is the exponential map $\Exp$. It is noted that given a curve $\gamma$ and $\eta\in \TM{\gamma(0)}$, there exists a unique parallel vector field $X_\eta$ along $\gamma$ such that $X_\eta(0)=\eta$. We define the parallel transport along $\gamma$ as $\cP_\gamma^{0\to t}\eta = X_\eta(t)$. If $\gamma$ is the geodesic, we know $\cP_\gamma^{0\to t}\dot\gamma(0) = \dot\gamma(t)$. If the minimizing geodesic $\gamma$ between $x$ and $y$ in $\cM$ is unique, we say $\cP_x^y$ is the parallel transport from $x$ to $y$ along the geodesic. We highlight that the parallel transport $\cP_\gamma^{0\to t}$ is a linear isometry, i.e., $\cP_\gamma^{0\to t}$ is linear and satisfies $\inprod{\cP_\gamma^{0\to t}{\eta_x}}{\cP_\gamma^{0\to t}\zeta_x} = \inprod{\zeta_x}{\xi_x}$ for all $\zeta_x,\xi_x\in\TM{x}$.

Let $\varphi$ be a smooth function on $\cM$, and we define $\grad \varphi$ to be the Riemannian gradient of $\varphi$. Moreover, if $G:\cM\rightarrow\cN$ is a smooth map between two Riemannian manifolds $\cM,\cN$, we use $\rmd G|_{x}:\TM{x}\mapsto T_{G(x)}\cN$ to denote the differential of $G$ at $x\in\cM$.

In the following, we list some results on Riemannian manifold $\cM$ that are useful in the convergence analysis. We assume $\cU$ is a compact subset of $\cM$.

\begin{lem}[\cite{daniilidis2018self}, Lemma 2.3]\label{lemma:1}
There exists a constant $\tilr > 0$, such that for any $x\in\cU$, $\Exp_x^{-1}:\cB_{\cM}(x,2\tilr)\rightarrow \cB_{\TM{x}}(0,2\tilr)$ is a diffeomorphism. 
\end{lem}

\begin{lem}\cite[Lemma 4.2]{zhou2021semi}\label{lem:basic}
There exist positive constants $C > 0$ and constant $\tilr > 0$, such that for all $x\in\cU$ and $v_x\in B_{\TM{x}}(0,\tilr)$, it has
\[
\dist\left(\cR_x(v_x),x\right) \leq C\|v_x\| \textand \dist\left(\cR_x(v_x),\Exp_x(v_x)\right) \leq C\|v_x\|^2.
\]
\end{lem}
\begin{lem}\cite[Lemma 2.4]{daniilidis2018self}\label{lem:coslaw}
There exists a constant $K > 0$ such that $\forall x,y,z\in\cU, y,z\in\cB_{\cM}(x,\tilr)$,
\[
\left|\|\eta_x-\xi_x\|^2-\dist(y,z)^2 \right| \leq K\|\xi_x\|^2\|\eta_x\|^2, 
\]
where $\eta_x := \Exp^{-1}_x y$ and $\xi_x := \Exp^{-1}_x z$.
\end{lem}

\begin{lem}\cite[Lemma 6]{sun2019escaping}\label{lem:2parallel}
    If the sectional curvature of $\cM$ is bounded, then there exists a constant $\rho_1$ such that for any $x,y,z\in\cM, \dist(x,y), \dist(y,z), \dist(x,z) < 2\tilr$, $\xi_x\in\TM{x}$, it has
    \[
    \left\| \cP_y^z\cP_x^y\ \xi_x-\cP_x^z\ \xi_x \right\| \leq \rho_1 \dist(y,z)\dist(x,y)\|\xi_x\|.
    \]
\end{lem}

\begin{lem}\cite[Lemma 3]{sun2019escaping}\label{lem:triangle_rule}
    If the sectional curvature of $\cM$ is bounded, then there exists a constant $c_0$ such that  $\forall x\in\cM, \xi_x,\eta_x\in\TM{x}, \|\eta_x\| < 2\tilr$, we have
    \[
    \dist(\Exp_x(\xi_x+\eta_x),\Exp_y(\cP_x^y\ \xi_x))
    \leq c_0 \min\{\|\xi_x\|,\|\eta_x\|\}(\|\xi_x\|+\|\eta_x\|)^2,
    \]
    where $y = \Exp_x(\eta_x)$.
\end{lem}

\section{The Riemannian Anderson Mixing (RAM) Algorithm}
\label{sec:alg}
\begin{algorithm}
    \caption{Riemannian Anderson Mixing Method}
    \label{alg:RAM}
    \begin{algorithmic}
        \REQUIRE $x_0 \in\cM, \epsilon, \beta_k > 0, m\in\N^*, k = 1$.
        \STATE $r_0 = - F(x_0)$.
        \STATE $x_1 = \cR_{x_0}(r_0)$.
        \WHILE{$\|F(x_k)\| \geq \epsilon$}
        \STATE $m_k = \min\set{m,k}$.
        \STATE $\dx{k-i}{k} = \cT_{x_{k-1}}^{x_k}\dx{k-i}{k-1} \in\TM{x_k}, i = 1,\ldots,m_k$.
        \STATE $\dr{k-1}{k} = r_k-\cT_{x_{k-1}}^{x_k}\ r_{k-1}\in\TM{x_k}$.
        \IF{$k \geq 2$}
        \STATE $\dr{k-i}{k} = \cT_{x_{k-1}}^{x_k}\dr{k-i}{k-1} \in\TM{x_k}, i = 2,\ldots,m_k$.
        \ENDIF
        \STATE $X_k = [\dx{k-m_k}{k},\ldots,\dx{k-1}{k}], R_k = [\dr{k-m_k}{k},\ldots,\dr{k-1}{k}]$.
        \STATE $r_k = -F(x_k)$.
        \STATE $\Gamma_k = \argmin_{\Gamma\in\R^{m_k}}\|r_k-R_k\Gamma\|$.
        \STATE $\Bar{r}_k = r_k-R_k\Gamma_k\in\TM{x_k}, \dx{k}{k} = -X_k\Gamma_k+\beta_k\Bar{r}_k \in\TM{x_k}$.
        \STATE $x_{k+1} = \cR_{x_k}(\dx{k}{k})$.
        \STATE $k = k+1$.
        \ENDWHILE
    \end{algorithmic}
\end{algorithm}
In this section, we provide a detailed description of the Riemannian Anderson mixing (RAM) algorithm and its convergence analysis. Throughout the paper, we define $r_k = -F(x_k)\in \TM{x_k}$, where $x_k\in\cM$ denotes the $k$-th iteration point on a complete, connected manifold $\cM$. At each iteration, we apply vector transport to each vector $\dx{k-i}{k-1}$ and $\dr{k-i}{k-1}$ from the tangent space at $x_{k-i}$ to the tangent space at $x_k$, denoted by $\dx{k-i}{k}$ and $\dr{k-i}{k}$, respectively. This leads to two matrices $X_k = [\dx{k-m_k}{k},\ldots,\dx{k-1}{k}]$ and $R_k = [\dr{k-m_k}{k},\ldots,\dr{k-1}{k}]$ in $\TM{x_k}$, where $m_k$ is the memory length at the $k$-th iteration. Let $\|\cdot\|$ be the norm in $\TM{x_k}$. We determine the extrapolation coefficients $\Gamma_k$ by solving the optimization problem $\Gamma_k = \argmin\|r_k-R_k\Gamma\|$. Using the extrapolation coefficients $\Gamma_k$, we compute a new direction $\dx{k}{k} = \beta_k r_k - (X_k+\beta_kR_k)\Gamma_k\in\TM{x_k}$. Finally, we obtain a new point $x_{k+1} = \cR_{x_k}(\dx{k}{k})$ by applying a predefined retraction $\cR$ on $\cM$. The detailed RAM method is presented in Algorithm \ref{alg:RAM}.

The RAM algorithm computes the update direction in the tangent space of the current point by utilizing an adaptive linear combination of previous iterations. This adaptive extrapolation is similar to classical AM in Euclidean space, exploiting second-order information and accelerating convergence. Each iteration in RAM requires $2m_k$ vector transports and one retraction, in addition to solving a least-squares problem in the tangent space. A significant advantage of RAM is that it avoids the need to compute the fixed-point map $g$, as well as the inverse of exponential mapping or retraction, which can be expensive in practice.

\subsection{Assumptions}
In this section, we introduce the basic assumptions before proving the local convergence analysis of RAM.  
\begin{asmp}\label{asmp:basic}
$\cM$ is a complete and connected manifold with bounded sectional curvature and $\cU$ is a compact subset of $\cM$. Moreover, there exist some $\tilr>0$ and a neighborhood $\cU$ on $\cM$ such that the open ball $\cB_\cM(x^*,2\tilr)\subset \cU$ and $F(x^*)=0$ for some $x^*\in\cU$.
\end{asmp}
The following assumption provides a uniformly lower bound of the distance and the uniform boundedness of the vector transport in the compact subset $\cU$.
\begin{asmp}\label{asmp:retr}
There exist a retraction $R_x$ and a constant $\Tilde{C} > 0$ such that for all $x\in\cU$ and $\|v_x\| < \tilr$, we have
    \[\dist\left(\cR_x(v_x),x\right) \geq \Tilde{C}\|v_x\|.\]
\end{asmp}
\begin{rmk}
    The \cref{asmp:retr} holds naturally when the retraction $\cR$ is exactly the exponential map. As pointed out in \cite[Lemma 6]{ring2012optimization}, this uniform lower bound exists when the retraction has equicontinuous derivatives. In addition, this lower bound can be easily derived through the direct calculation for the case when $\cM$ is the sphere $\cS^{n-1}$, and the retraction is given by $\cR_x(v) = \frac{x+v}{\|x+v\|}$. For completeness, we give the proof in the supplementary material.
\end{rmk}
\begin{asmp}\label{asmp:unibound_vt}
    The vector transport $\cT$ is a $C^1$-mapping and there exists a constant $M_\cU > 0$ such that for any $x,y\in\cU$ (whenever $\cT_{x}^y$ is well defined), and $v_x\in\TM{x}$, we have $\|\cT_x^y v_x\|_y \leq M_\cU \|v_x\|_x$.
\end{asmp}
\begin{rmk}
The \cref{asmp:unibound_vt} holds when the vector transport is provided by parallel transport.
\end{rmk}
Similar to the Euclidean case in \cite{toth2015convergence,wei2021class}, we impose the "coercive" and "locally Lipschitz continuous" property on the vector field $F$ and the contractive property on the fixed point map $g(x):= \Exp_x(-F(x))$.
\begin{asmp}\label{asmp:F_lip}
There exist $0 < L_1 < L_2$ such that for all $x,y\in\cB_{\cM}(x^*, 2\tilr)$ such that $\dist(x,y) < 2\tilr$, we have
\[
L_1\dist(x,y) \leq \|F(x)-\cP_y^x F(y)\| \leq L_2\dist(x,y).
\]
\end{asmp}
\begin{asmp}\label{asmp:contraction}
There exists $0<\kappa<1$ such that for all $x,y\in\cU$, we have
\[
\dist(g(x),g(y)) \leq \kappa\dist(x,y).
\]
\end{asmp}
Define $H$ to be the Jacobian of the vector field $F$, and we further impose Lipschitz continuous assumption on $H$.
\begin{asmp}\label{asmp:jacobian}
There exits $L > 0$, such that for all $x,y\in\cU$, we have 
\[
\left\| \cP_y^x H(y) \cP_x^y - H(x) \right\| \leq L \dist(x,y).
\]
\end{asmp}
The final assumption concerns the uniform boundedness of the extrapolation coefficients $\Gamma_k$ in the RAM method. This assumption is commonly used for analyzing convergence in the Euclidean setting \cite[Assumption 2.1]{toth2015convergence}.  
\begin{asmp}\label{asmp:bounded_least_square}
There exists a positive constant $M_{\Gamma}$ such that  $\|\Gamma_k\|_{\infty} \leq M_{\Gamma}$ for all $ k\in\N$.
\end{asmp}

\subsection{Useful lemmas}
We give some results related to the iteration property of RAM that is needed for the convergence analysis. Due to the space limit, the proof of lemmas is given in the supplemental material.
The first lemma is an error estimation between the parallel transport $\cP$ and the vector transport $\cT$.
\begin{lem}\label{lem:para_vect_error}
Suppose \cref{asmp:retr} holds. There exists a constant $\rho>0$, such that for all $x,y\in\cU$ satisfying $\dist(x,y) < 2\tilr$ and $\|\cR_x^{-1}y\| < \tilr$, and $\forall v_x\in\TM{x}$, we have
\[
\|\cP_x^y\ v_x-\cT_x^y\ v_x\| \leq \rho\|v_x\|\dist(x,y)
\]
\end{lem}

There are multiple vector transport in RAM, and we need to characterize their deviations from parallel translations based on the above assumptions. We present two corollaries from \cref{lem:para_vect_error} and \cref{lem:2parallel}.
\begin{coro}\label{coro:multi_para_vect_err}
Suppose $x_1,x_2,\ldots,x_n\in\cU$ satisfying 
$\dist(x_i,x_{i+1}) < 2\tilr$ and $\|\cR_{x_i}^{-1}x_{i+1}\| < \tilr, 1\leq i\leq n-1$, $n\geq 2$ and 
that \cref{asmp:basic} ,\cref{asmp:retr} and \cref{asmp:unibound_vt} hold. If $v\in\TM{x_1}$, then
\[
\begin{aligned}
&\|\cP_{x_{n-1}}^{x_n}\cP_{x_{n-2}}^{x_{n-1}}\cdots\cP_{x_{1}}^{x_2}\ v-\cT_{x_{n-1}}^{x_n}\cT_{x_{n-2}}^{x_{n-1}}\cdots\cT_{x_{1}}^{x_2}\ v\| 
\leq \rho\|v\|\sum_{i=1}^{n-1}M_{\cU}^{i-1}\dist(x_i,x_{i+1}).   
\end{aligned}
\]
\end{coro}
\begin{coro}\label{coro:multi_para}
Suppose \cref{asmp:basic} holds and $x_1,x_2,\ldots,x_n\in\cM$, $\dist(x_1,x_i), \dist(x_{i-1},x_i) < 2\tilr, 2\leq i \leq n$, $n> 2$. If $\xi\in\TM{x_1}$, then
\[
\left\| \cP_{x_{1}}^{x_{n}}\ \xi - \cP_{x_{n-1}}^{x_{n}}\cP_{x_{n-2}}^{x_{n-1}}\cdots\cP_{x_{1}}^{x_{2}}\ \xi \right\| \leq \rho_1\|\xi\|\sum_{i=2}^{n-1}\dist(x_i,x_{i+1})\dist(x_1,x_i).
\]
\end{coro}
The next lemma is an analogy of \cref{lem:2parallel} for the Jacobian of $F$:
\begin{lem}\label{lem:para_jacobian}
    If \cref{asmp:jacobian} holds, then for all $x,y,z \in\cU$ and $\dist(x,y),\dist(y,z), \dist(x,z) < 2\tilr$, we have 
    \[
    \norm{\cP_{y}^x\cP_{z}^y H(z)\cP_{y}^z\cP_{x}^y - \cP_{z}^x H(z) \cP_{x}^z} \leq 2\rho_1\dist(x,y)\dist(y,z)\norm{H(z)}.
    \]
\end{lem}

Given arbitrary $x_1,\ldots,x_k\in\cU (k\leq m)$ satisfying $\|\cR^{-1}_{x_{k-j}}(x_{k-j+1})\| < \tilr$ , define $\dx{k-j}{k} =\cT_{x_{k-1}}^{x_k}\cdots\cT_{x_{k-j}}^{x_{k-j+1}}\cR^{-1}_{x_{k-j}}(x_{k-j+1})$. The following lemma provides an upper bound of the norm of the iterative direction $\dx{k}{k}$ in terms of  $\|r_{k-j}\|, j = 0,\ldots,k$.  
\begin{lem}\label{lem:err_estimate} 
Suppose \cref{asmp:basic}, \cref{asmp:retr}, \cref{asmp:unibound_vt}, \cref{asmp:F_lip} and \cref{asmp:bounded_least_square} hold. If the points $x_1,\ldots,x_k\in\cB_{\cM}(x^*,2\tilr) (k\leq m)$ satisfy $\dist(x_{k-j},x_{k-j+1}) < 2\tilr$ and $\|\cR^{-1}_{x_{k-j}}(x_{k-j+1})\| < \tilr, 1 \leq j \leq k$, there exists a constant $M_1>0$, irrelevant to the choice of $x_1,\ldots,x_k$, such that
\begin{equation}
    \|X_k\Gamma_k\|+\|r_k\| \leq M_1\sum_{i=0}^{k}\|r_{k-i}\| = \sum_{i=0}^{k}\cO(\|r_{k-i}\|),
\end{equation}
and
\begin{equation}
     \sum_{j=1}^i|\gamma_j^k|\|\dx{k-j}{k}\|, \sum_{j=1}^i\|\dx{k-j}{k}\|\leq M_1\sum_{j=0}^i\|r_{k-i}\| = \sum_{j=0}^i\cO(\|r_{k-i}\|), \forall 1\leq i\leq k.
\end{equation}
\end{lem}

Define $r^\prime = \min\{\frac{\tilr}{2(mM_1+1)L_2},\frac{\tilr}{4}\}$. Next corollary ensures that all the auxiliary points based on $x_k$ in the proof of following propositions and theorems fall into a small neighborhood of $x_k$.
\begin{coro}\label{coro:dist_control}
Suppose \cref{asmp:basic}, \cref{asmp:retr}, \cref{asmp:unibound_vt}, \cref{asmp:F_lip} and \cref{asmp:bounded_least_square} hold. Assume the points $x_1,\ldots,x_k\in\cB_{\cM}(x^*,r^\prime) (k\leq m)$ satisfy $\|\cR^{-1}_{x_{j}}(x_{j+1})\| < \tilr, 1 \leq j \leq k-1$. Fix any integer $1\leq i\leq k$ and real number $0\leq \alpha,\beta \leq 1$. Choose $v\in\TM{x_k}$ such that $\|v\| \leq \|r_k\|$. Define $\eta := -\alpha\sum_{j=1}^i \gamma_j^k \dx{k-j}{k} +\beta v$ and $\xi := -\alpha\sum_{j=1}^i \dx{k-j}{k} +\beta v$. Then we have $\|\eta\|, \|\xi\| < \frac{\tilr}{2}$. In particular, if we define $y = \Exp_{x_k}(\eta)$ and $z = \Exp_{x_k}(\xi)$, then $x_1,\ldots,x_{k-1},y,z\in\cB_{\cM}(x_k,\frac{\tilr}{2})$.
\end{coro}

In the Euclidean space, $x_k-x_{k-i}$ can be written as the sum of the difference of two consecutive iterations, i.e., $x_k-x_{k-i} = \sum_{j=0}^{i-1} (x_{k-j}-x_{k-j-1})$. In contrast, multiple transportations bring deviations in the Riemannian manifold, which are bounded in the next proposition.

\begin{prop}\label{prop:multi_step_err}
Suppose \cref{asmp:basic}, \cref{asmp:retr}, \cref{asmp:unibound_vt}, \cref{asmp:F_lip} and \cref{asmp:bounded_least_square} hold. If the points $x_1,\ldots,x_k\in\cB_{\cM}(x^*,r^\prime) (k\leq m)$ satisfy $\|\cR^{-1}_{x_{k-j}}(x_{k-j+1})\| < \tilr$ for $1\leq j \leq k-1$, there exists a constant $M_2 > 0$ irrelevant to the choice of $x_1,\ldots,x_k$ such that for any $1\leq i\leq k$, we have
\[
\left\| \sum_{j=1}^i\dx{k-j}{k}+\Exp_{x_k}^{-1}x_{k-i} \right\| \leq 
 M_2\sum_{j=0}^i\|r_{k-j}\|^2 = \sum_{j=0}^i\cO(\|r_{k-j}\|^2).
\]
\end{prop}

Suppose $\Gamma_k = (\gamma_1^k,\ldots,\gamma_{k}^k)$ and set $\gamma_0^k = 0$. Let $z_k^i := \sum_{j=1}^i\dx{k-j}{k}$ and $y_k^i:= \Exp_{x_k}\left( -z_k^i \right)$ for $1\leq i\leq k$,
the following proposition estimates the distance between $F(y_k^i)$ and $F(x_{k-i})$ after multiple parallel transportations. 

\begin{prop}\label{prop:2}
Suppose \cref{asmp:basic}, \cref{asmp:retr}, \cref{asmp:unibound_vt}, \cref{asmp:F_lip} and \cref{asmp:bounded_least_square} hold. If the points $x_1,\ldots,x_k\in\cB_{\cM}(x^*,r^\prime) (k\leq m)$ satisfy $\|\cR^{-1}_{x_{k-j}}(x_{k-j+1})\| < \tilr$ , $1\leq j \leq k-1$, there exists a constant $M_3 > 0$ irrelevant to the choice of $x_1,\ldots,x_k$ such that for any $1\leq i \leq k$ we have
\[
\left\| \cP_{y_k^i}^{x_k}F(y_k^i)-\cT_{x_{k-1}}^{ x_k}\cT_{x_{k-2}}^{x_{k-1}}\cdots\cT_{x_{k-i}}^{x_{k-i+1}}F(x_{k-i}) \right\| \leq M_3\sum_{j=0}^i\|r_{k-j}\|^2 = \sum_{j=0}^i\cO(\|r_{k-j}\|^2).
\]
\end{prop}

Let $w_k^i := \sum_{j=0}^{i}\gamma_j^k\dx{k-j}{k}$ and $v_k^i = \Exp_{x_k}(-w_k^i)$. Next proposition is necessary for the proof of \cref{thm:RAM_local_conv}.
\begin{prop}\label{prop:3}
    Suppose \cref{asmp:F_lip}, \cref{asmp:jacobian} and \cref{asmp:bounded_least_square} hold. If the points $x_1,\ldots,x_k\in\cB_{\cM}(x^*,r^\prime)\subset\cU (k\leq m)$ satisfy $\|\cR^{-1}_{x_{k-j}}(x_{k-j+1})\| < \tilr$, then there exists a constant $M_4 > 0$ irrelevant to the choice of $x_1,\ldots,x_k$ such that for all $1\leq i \leq k$ we have
    \[
        \norm{\cP_{v_k^i}^{x_k} F(v_k^i)-\cP_{v_k^{i-1}}^{x_k} F(v_k^{i-1})-\gamma_i^k (\cP_{y_k^i}^{x_k} F(y_k^i)-\cP_{y_k^{i-1}}^{x_k} F(y_k^{i-1}))} \leq M_4\sum_{j=0}^i\|r_{k-j}\|^2.
    \]
\end{prop}

\subsection{Locally Linear Convergence of RAM}
Setting 
$$r = \min\left\{ r^\prime, \frac{2\tilr}{1+CM_1L_2m},\frac{2\tilr}{L_2(1+CM_1L_2m)},\frac{\tilr}{2(1+M_1L_2^2m)} \right\},$$
and $x_{k+1} = \mathrm{RAM}(x_1,\ldots,x_k)$ that means the output of the one-step iteration of \cref{alg:RAM} based on $x_1,\cdots,x_k$. 
\begin{thm}\label{thm:RAM_local_conv}
Suppose \cref{asmp:basic}-\cref{asmp:bounded_least_square} hold. Given arbitrary $x_1,\ldots,x_k\in\cB_{\cM}(x^*,r)\subset \cU (k\leq m)$ satisfying $\|\cR^{-1}_{x_{i}}(x_{i+1})\| < \tilr, 1\leq i\leq k-1$ and set $x_{k+1} = \mathrm{RAM}(x_1,\ldots,x_k)$, there exists a constant $\widehat{M}>0$, irrelevant to the choice of $x_1,\ldots,x_k$, such that
\begin{equation}\label{eqn:RAM_conv}
     \begin{aligned}
     \|r_{k+1}\| \leq& \theta_k\left[(1-\beta_k)+\kappa\beta_k\right]\|r_k\|+\sum_{i=0}^{k}\cO(\|r_{k-i}\|^2) 
     \leq \theta_k\left[(1-\beta_k)+\kappa\beta_k\right]\|r_k\|+\widehat{M}\sum_{i=0}^{k}\|r_{k-i}\|^2,
    \end{aligned} 
\end{equation}
where $\theta_k = \frac{\|\Bar{r}_k\|}{\|r_k\|} \leq 1$.
\end{thm} 

\begin{proof}
Since $x_1,\ldots,x_k\in\cB_{\cM}(x^*,r)$ and $r< r^\prime$, we have $\forall 0\leq i,j \leq k, \dist(x_{i},x_{j})\leq \dist(x_{i},x^*)+\dist(x_{j},x^*) < 2r < \tilr$. Define
\[
\left\{
\begin{aligned}
&\Bar{x}_k:=\Exp_{x_k}(-X_k\Gamma_k), \\
&x_{k+1}^\prime:= \Exp_{x_k}(-X_k\Gamma_k+\beta_k\Bar{r}_k), \\
&\Tilde{x}_{k+1}:= \Exp_{x_k}(-X_k\Gamma_k+\Bar{r}_k).
\end{aligned}\right.
\]

Notice that $\beta_k \leq 1$ and $\|\Bar{r}_k\| \leq \|r_k\|$ by definition. Therefore \cref{coro:dist_control} implies that $x^\prime_{k+1},\Bar{x}_{k},\Tilde{x}_{k+1} \in \cB_{\cM}(x_k,\tilr/2)$. Since by \cref{lem:basic} ,\cref{lem:err_estimate} and \cref{asmp:F_lip} , $\dist(x_{k+1},x^*) \leq \dist(x_{k},x^*)+\dist(x_{k+1},x_k)  \leq (1+CM_1L_2m)r < 2\tilr$. So
$\|r_{k+1}\| \leq L_2\dist(x_{k+1},x^*) \leq  L_2(1+CM_1L_2m)r < 2\tilr$
 , then the equation $\|r_{k+1}\| = \|-F(x_{k+1})\| = \dist\left(\Exp_{x_{k+1}}(-F(x_{k+1})),x_{k+1}\right)$ holds.

We then have:
\begin{equation}
    \begin{aligned}
    \|r_{k+1}\| = \dist(g(x_{k+1}),x_{k+1}) \leq \dist(g(x_{k+1}),g(x_{k+1}^\prime))+\dist(g(x_{k+1}^\prime),g(\Bar{x}_{k}))+\dist(g(\Bar{x}_k),x_{k+1}).
    \end{aligned}
\end{equation}
Since $g$ by \cref{asmp:contraction} is a contraction and $\|\dx{k}{k}\| < \frac{\tilr}{2}$ by \cref{coro:dist_control}, we know from \cref{lem:basic} and \cref{lem:err_estimate} that
\begin{equation}
     \begin{aligned}
      \dist(g(x_{k+1}),g(x_{k+1}^\prime)) \leq& \kappa\dist\left(\cR_{x_k}(\dx{k}{k}),\Exp_{x_k}(\dx{k}{k})\right) 
      \leq \kappa C\|\dx{k}{k}\|^2 \\
      \leq& \kappa C M_1^2(\sum_{i=0}^{k}\|r_{k-i}\|)^2 
      \leq \kappa m CM_1^2\sum_{i=0}^{k}\|r_{k-i}\|^2.
     \end{aligned}
\end{equation}
We also have by \cref{lem:coslaw} and \cref{lem:err_estimate} that
\begin{equation}
    \begin{aligned}
     \dist(g(x_{k+1}^\prime),g(\Bar{x}_{k})) \leq& \kappa\dist\left( \Exp_{x_k}(-X_k\Gamma_k),\Exp_{x_k}(-X_k\Gamma_k+\beta_k\Bar{r}_k) \right) \\
     \leq& \kappa\left[ \beta_k\|\Bar{r}_k\|+\sqrt{K}\|-X_k\Gamma_k\|\|-X_k\Gamma_k+\beta_k\Bar{r}_k\| \right] \\
     \leq& \kappa\beta_k\theta_k\|r_k\|+\kappa\sqrt{K}(\|X_k\Gamma_k\|+\|r_k\|)^2 \\
     \leq& \kappa\beta_k\theta_k\|r_k\| + \kappa m \sqrt{K}M_1^2\sum_{i=0}^{k}\|r_{k-i}\|^2.
    \end{aligned}
\end{equation}
But
\begin{equation}
    \dist(x_{k+1},g(\Bar{x}_k)) \leq \dist(g(\Bar{x}_k),\Tilde{x}_{k+1})+\dist(\Tilde{x}_{k+1},x_{k+1}^\prime)+\dist(x_{k+1}^\prime,x_{k+1}).
\end{equation}
Notice that by \cref{lem:basic} and \cref{lem:err_estimate}
\[
\dist(x_{k+1}^\prime,x_{k+1}) \leq C\|\dx{k}{k}\|^2 \leq m CM_1^2\sum_{i=0}^{k}\|r_{k-i}\|^2,
\]
and by \cref{lem:coslaw} and \cref{lem:err_estimate}
\[
\begin{aligned}
\dist(\Tilde{x}_{k+1},x_{k+1}^\prime) &\leq (1-\beta_k)\theta_k\|r_k\|+\sqrt{K}\|-X_k\Gamma_k+\Bar{r}_k\| \| -X_k\Gamma_k+\beta_k \Bar{r}_k \|  \\
&\leq (1-\beta_k)\theta_k\|r_k\|+m \sqrt{K}M_1^2\sum_{i=0}^{k}\|r_{k-i}\|^2.
\end{aligned}
\]
Now we have:
\begin{equation}
    \begin{aligned}
        \|r_{k+1}\| \leq \theta_k\left[(1-\beta_k)+\kappa\beta_k\right]\|r_k\|+\dist(g(\Bar{x}_k),\Tilde{x}_{k+1})+(\kappa+1)(C+\sqrt{K})mM_1^2\sum_{i=0}^{k}\|r_{k-i}\|^2.
    \end{aligned}   
\end{equation}

Thus it suffices to estimate $\dist(g(\Bar{x}_k),\Tilde{x}_{k+1})$. Note that
\[
\dist(g(\Bar{x}_k),\Tilde{x}_{k+1}) \leq \dist(g(\Bar{x}_k),\Exp_{\Bar{x}_k}(\cP_{x_k}^{\Bar{x}_k}\Bar{r}_k)) + \dist(\Exp_{\Bar{x}_k}(\cP_{x_k}^{\Bar{x}_k}\Bar{r}_k),\Tilde{x}_{k+1}).
\]
\cref{lem:triangle_rule} and  \cref{lem:err_estimate} yield:
\[
\begin{aligned}
    \dist(\Exp_{\Bar{x}_k}(\cP_{x_k}^{\Bar{x}_k}\Bar{r}_k),\Tilde{x}_{k+1}) 
    = &\dist(\Exp_{\Bar{x}_k}(\cP_{x_k}^{\Bar{x}_k}\Bar{r}_k),\Exp_{x_k}(-X_k\Gamma_k+\Bar{r}_k)) \\
    \leq & c_0\min\{ \|r_k\|, \|-X_k\Gamma_k\| \}(\|r_k\|+\|-X_k\Gamma_k\|)^2 \\
    \leq & \frac{c_0 \tilr}{2} m M_1^2\sum_{i=0}^{k}\|r_{k-i}\|^2.
\end{aligned}
\]

Notice that $\|F(\Bar{x}_k)\| \leq \|r_k\|+L_2\dist(x_{k},\Bar{x}_k) \leq (1+M_1L_2^2m)r < \frac{\tilr}{2}$, and it implies $\dist(g(\Bar{x}_k),x_k) < \tilr$. Together with \cref{lem:coslaw}, we obtain
\begin{equation}
    \begin{aligned}
    \dist(g(\Bar{x}_k),\Exp_{\Bar{x}_k}(\cP_{x_k}^{\Bar{x}_k}\Bar{r}_k)) \leq& \left\| -F(\Bar{x}_k)-\cP_{x_k}^{\Bar{x}_k}\Bar{r}_k \right\| + \sqrt{K}\|F(\Bar{x}_k)\|\|\Bar{r}_k\| \\
    \leq& \left\| F(\Bar{x}_k)+\cP_{x_k}^{\Bar{x}_k}\Bar{r}_k \right\| + \sqrt{K}(\|r_k\|+L_2\|X_k\Gamma_k\|)\|r_k\| \\
    \leq& \left\| F(\Bar{x}_k)+\cP_{x_k}^{\Bar{x}_k}\Bar{r}_k \right\| + \sqrt{K}\max\{1,L_2\}(\|r_k\|+\|X_k\Gamma_k\|)^2 \\
    \leq& \left\| F(\Bar{x}_k)+\cP_{x_k}^{\Bar{x}_k}\Bar{r}_k \right\| + \sqrt{K}\max\{1,L_2\}mM_1^2\sum_{i=0}^{k}\|r_{k-i}\|^2. \\
    \end{aligned}
\end{equation}
It remains to estimate $\left\| F(\Bar{x}_k)+\cP_{x_k}^{\Bar{x}_k}\Bar{r}_k \right\|$.  Set $\sum_{j=1}^0 = 0$. Let $w_k^i := \sum_{j=0}^{i}\gamma_j^k\dx{k-j}{k}$ and $z_k^i := \sum_{j=1}^{i}\dx{k-j}{k}$. Denote $v_k^i = \Exp_{x_k}(-w_k^i)$ and $y_k^i = \Exp_{x_k}(-z_k^i)$. Obviously $y_k^0 = v_k^0 = x_k$ and $v_k^{k} = \Bar{x}_k$. Thus, it has
\begin{equation}
    \begin{aligned}
        \left\| F(\Bar{x}_k)+\cP_{x_k}^{\Bar{x}_k}\Bar{r}_k \right\| 
        & = \left\| \cP_{\Bar{x}_k}^{x_k}F(\Bar{x}_k)+\Bar{r}_k \right\| 
        =\ \left\|
        \cP_{\Bar{x}_k}^{x_k}F(\Bar{x}_k)- F(x_k)-\sum_{i=1}^{k}\gamma_i^k\dr{k}{k-i} \right\| \\
        \leq\ &\sum_{i=1}^{k}\norm{\cP_{v_k^i}^{x_k} F(v_k^i)-\cP_{v_k^{i-1}}^{x_k} F(v_{k}^{i-1})-\gamma_i^k \left(\cP_{y_k^i}^{x_k} F(y_k^i)-\cP_{y_k^{i-1}}^{x_k} F(y_{k}^{i-1})\right)} \\
        +&\sum_{i=1}^{k}|\gamma_i^k| \left\| -\dr{k-i}{k}+\cP_{y_k^i}^{x_k}F(y_k^i)-\cP_{y_k^{i-1}}^{x_k}F(y_k^{i-1}) \right\|.
    \end{aligned}
\end{equation}
Applying \cref{prop:3}, we have for all $1\leq i \leq k$,
\begin{equation}
    \norm{\cP_{v_k^i}^{x_k} F(v_k^i)-\cP_{v_k^{i-1}}^{x_k} F(v_{k}^{i-1})-\gamma_i^k \left(\cP_{y_k^i}^{x_k} F(y_k^i)-\cP_{y_k^{i-1}}^{x_k} F(y_{k}^{i-1}\right)} \leq M_4\sum_{j=0}^i\|r_{k-j}\|^2.
\end{equation}
On the other hand,
\begin{equation}
    \begin{aligned}
    &|\gamma_i^k| \left\| \dr{k-i}{k}-\cP_{y_k^i}^{x_k}F(y_k^i)+\cP_{y_k^{i-1}}^{x_k}F(y_k^{i-1}) \right\| \\
    \leq& M_\Gamma \left\| \cT_{x_{k-1}}^{x_k}\cdots\cT_{x_{k-i+1}}^{x_{k-i+2}}\ r_{k-i+1}+\cP_{y_k^{i-1}}^{x_k}F(y_k^{i-1}) \right\|  + M_\Gamma \left\| \cT_{x_{k-1}}^{x_k}\cdots\cT_{x_{k-i}}^{x_{k-i+1}}\ r_{k-i}+\cP_{y_k^{i}}^{x_k}F(y_k^i) \right\|. 
    \end{aligned}
\end{equation}
 According to \cref{prop:2}, we know 
\[
\left\| \cT_{x_{k-1}}^{x_k}\cdots\cT_{x_{k-i}}^{x_{k-i+1}}\ r_{k-i}+\cP_{y_k^{i}}^{x_k}F(y_k^i) \right\| \leq M_3\sum_{j=0}^i\|r_{k-j}\|^2 \leq M_3\sum_{j=0}^k\|r_{k-j}\|^2,
\]
and
\[
\left\| \cT_{x_{k-1}}^{x_k}\cdots\cT_{x_{k-i+1}}^{x_{k-i+2}}r_{k-i+1}+\cP_{y_k^{i-1}}^{x_k}F(y_k^{i-1}) \right\| \leq M_3\sum_{j=0}^{i-1}\|r_{k-j}\|^2  \leq M_3\sum_{j=0}^k\|r_{k-j}\|^2.
\]
This implies for $1\leq i\leq k$,
\[
|\gamma_k^i| \left\| \dr{k-i}{k}-\cP_{y_k^i}^{x_k}F(y_k^i)+\cP_{y_k^{i-1}}^{x_k}F(y_k^{i-1}) \right\| \leq 2M_{\Gamma}M_3\sum_{j=0}^k\|r_{k-j}\|^2.
\]
Finally, we obtain
\begin{equation}
    \begin{aligned}
        \left\| F(\Bar{x}_k)+\cP_{x_k}^{\Bar{x}_k}\Bar{r}_k \right\| 
         \leq\ &  \sum_{i=1}^{k}\norm{\cP_{v_k^i}^{x_k} F(v_k^i)-\cP_{v_k^{i-1}}^{x_k} F(v_{k}^{i-1})-\gamma_i^k \left(\cP_{y_k^i}^{x_k} F(y_k^i)-\cP_{y_k^{i-1}}^{x_k} F(y_{k}^{i-1})\right)} \\
        +&\sum_{i=1}^{k}|\gamma_k^i| \left\| -\dr{k-i}{k}+\cP_{y_k^i}^{x_k}F(y_k^i)-\cP_{y_k^{i-1}}^{x_k}F(y_k^{i-1}) \right\| \\
        \leq\ &\sum_{i=1}^{k}M_4\sum_{j=0}^{i}\|r_{k-j}\|^2 + \sum_{i=1}^{k}2M_{\Gamma}M_3\sum_{j=0}^k\|r_{k-j}\|^2 \\
        \leq\ &m[M_4+2M_{\Gamma}M_3]\sum_{j=0}^k\|r_{k-j}\|^2.
    \end{aligned}
\end{equation}
Therefore, setting \[\widehat{M} := \left[(\kappa+1)(C+\sqrt{K})+\frac{c_0 \tilr}{2}+\sqrt{K}\max\{1,L_2\}\right]mM_1^2+m\left[M_4+2M_{\Gamma}M_3\right],\] we know the inequality \eqref{eqn:RAM_conv} holds.
\end{proof}


Based on \cref{thm:RAM_local_conv}, we give the following convergence theorem of RAM. Choose $\beta_k$ such that $\beta: = \inf_k \beta_k > \frac{L_2-L_1}{(1-\kappa)L_2}$. Set $\hat{r} = \min\left\{r,\frac{(1+CL_2)\tilr}{L_2}, \frac{L_1-L_2+(1-\kappa)\beta L_2}{\widehat{M}mL_2^2}\right\}$.
\begin{thm}\label{thm:RAM_final_local_conv}
    Suppose $\set{x_n}$ is generated by Algorithm \ref{alg:RAM} and suppose \cref{asmp:basic}-\cref{asmp:bounded_least_square} hold. If the initial point $x_0\in\cM$ satisfies $\dist(x_0,x^*)<\frac{\hat{r}}{1+CL_2}$, then for all $k \geq 1$, we have $\dist(x_k,x^*) < \hat{r}$ and $\|\cR^{-1}_{x_{k-1}}x_{k}\| < \tilr$. Moreover, inequality \ref{eqn:RAM_conv} always holds in this case, i.e we have the following local linear convergence result:
\begin{equation}
     \|r_{k+1}\| \leq \theta_k\left[(1-\beta_k)+\kappa\beta_k\right]\|r_k\|+\sum_{i=0}^{m_k}\cO(\|r_{k-i}\|^2) 
\end{equation}
where $\theta_k = \frac{\|\Bar{r}_k\|}{\|r_k\|} \leq 1$.
\end{thm}

\begin{proof}
    We prove this by induction. For $k=1$, $\|\cR^{-1}_{x_0}x_1\| = \|r_0\| \leq L_2\dist(x_0,x^*) < \frac{L_2}{1+CL_2}\hat{r} < \tilr$. On the other hand, we have $\dist(x_1,x^*) \leq \dist(x_0,x^*)+\dist(x_1,x_0) < \frac{1}{1+CL_2}\hat{r}+C\|r_0\| < \hat{r}$. Suppose the theorem holds for all $j\leq k$. Since $\hat{r}\leq r$, by induction hypothesis and \cref{thm:RAM_local_conv}, we have
\[
\begin{aligned}
\|r_{k+1}\| \leq& \theta_k \left[(1-\beta_k)+\kappa \beta_k\right]\|r_k\| + \widehat{M}\sum_{j=0}^{m_k}\|r_{k-j}\|^2 \\
\leq& \left[(1-\beta_k)+\kappa \beta_k\right]L_2\hat{r}+\widehat{M}mL_2^2 \hat{r}^2 \\
\leq& \left[ (1-\beta_k)L_2+\kappa\beta_kL_2+\widehat{M}mL_2^2 \frac{L_1-L_2+(1-\kappa)\beta L_2}{\widehat{M}mL_2^2} \right]\hat{r} 
\leq L_1\hat{r}.
\end{aligned}
\]
Also, from the induction hypothesis and the beginning in the proof of \cref{thm:RAM_local_conv}, we already know that $\dist(x_{k+1},x^*) < 2\tilr$. Therefore by \cref{asmp:F_lip} $\dist(x_{k+1},x^*) \leq \frac{1}{L_1}\|r_{k+1}\| < \hat{r}$. Furthermore, by \cref{coro:dist_control} $\|\cR^{-1}_{x_{k}}x_{k+1}\| = \|\dx{k}{k}\| < \tilr/2 <\tilr$.  The proof is thus completed.
\end{proof}

\section{Regularied RAM and Global Convergence Analysis}
\label{sec:global_conv}
The global convergence of Anderson mixing remains an open problem, even in the Euclidean setting. In linear problems, Florian \cite{potra2013characterization} highlighted that the method may fail to converge to the correct solution if the initial point is not selected correctly. Although Anderson mixing has a close relationship with GMRES or multisecant methods, previous convergence theory for these methods cannot be applied to Anderson mixing. To investigate the global convergence of the Riemannian Anderson mixing (RAM) method, we propose slight modifications to the original algorithm, including adding another stepsize $\alpha_k$ in the iteration and regularization in the least-squares problem. Building on the work of \cite{wei2021stochastic}, we develop the Regularized Riemannian Anderson Mixing (RRAM) method, which is presented in \cref{alg:RRAM}.
\begin{algorithm}
    \caption{Regularized Riemannian Anderson Mixing Method}
    \label{alg:RRAM}
    \begin{algorithmic}[1]
        \REQUIRE $x_0 \in\cM, \epsilon, \beta_k,\delta_k > 0, \alpha_k \geq 0, m\in\N^*, k = 1$.
        \STATE $x_1 = \cR_{x_0}(r_0)$.
        \WHILE{$\|F(x_k)\| \geq \epsilon$}
        \STATE $r_k = -F(x_k)$.
        \STATE $m_k = \min\set{m,k}$.
        \STATE $\dx{k-i}{k} = \cT_{x_{k-1}}^{x_k}\dx{k-i}{k-1} \in\TM{x_k}, i = 1,\ldots,m_k$.
        \STATE $\dr{k-1}{k} = r_k-\cT_{x_{k-1}}^{x_k}\ r_{k-1}\in\TM{x_k}$.
        \IF{$k \geq 2$}
        \STATE $\dr{k-i}{k} = \cT_{x_{k-1}}^{x_k}\dr{k-i}{k-1} \in\TM{x_k}, i = 2,\ldots,m_k$.
        \ENDIF
        \STATE $X_k = [\dx{k-m_k}{k},\ldots,\dx{k-1}{k}], R_k = [\dr{k-m_k}{k},\ldots,\dr{k-1}{k}]$.
        \STATE $\Gamma_k = \argmin_{\Gamma\in\R^{m_k}}\|r_k-R_k\Gamma\|_{x_k}^2+\delta_k\|X_k\Gamma\|_{x_k}^2$.
        \STATE Choose $\alpha_k$ satisfying condition \eqref{eqn:alpha}.
        \STATE $\dx{k}{k} = \beta_k r_k-\alpha_k (X_k+\beta_k R_k)\Gamma_k  \in\TM{x_k}$.
        \STATE $x_{k+1} = \cR_{x_k}(\dx{k}{k})$.
        \STATE $k = k+1$.
        \ENDWHILE
    \end{algorithmic}
\end{algorithm}

We first examine the global convergence behavior of RRAM, which depends on the selection of the new parameters $\alpha_k$ and $\delta_k$. To facilitate our analysis, we make several basic assumptions regarding the vector field $F(x)$.

\begin{asmp}\label{asmp:global_base}
There exists a $C^2$ bounded below function defined on $\cM$ such that the vector field $F(x)$ is the Riemannian gradient of $f$. Namely there exists $f:\cM\rightarrow \R$ such that $f(x)\geq f_{\text{low}} > -\infty\quad \forall x\in\cM$ and that $F(x) = \grad f(x)$.
\end{asmp}

\begin{asmp}\label{asmp:global_Lip}
The Riemannian gradient of $f$ is Lipschitz continuous on $\cM$, i.e. there exists a constant $L_2$, such that for any $x,y\in\cM$ with well-defined $\cP_y^x$,
\[
\|\grad f(x) - \cP_y^x \grad f(y)\| \leq L_2 \dist(x,y).
\]
\end{asmp}

\begin{asmp}\label{asmp:compact_levelset}
The manifold $\cM$ and $C^2$ function $f$ satisfy that for any $x_0\in\cM$, the sublevel set $\cL := \set{x\in\cM: f(x) \leq f(x_0)}$ is always compact.
\end{asmp}

\cref{asmp:compact_levelset} and the following \cref{lem:R_Lip} help us secure the Lipschitz continuity of $\grad f(x)$.
\begin{lem}\cite[Lemma C.8]{levin2021finding}\label{lem:R_Lip}
Consider a retraction $\cR$ on $\cM$, a compact subset $\cL\subset\cM$ and a continuous, nonnegative function $h:\cL\rightarrow\R$. The set
\[
\cS := \set{(x,v)\in\tm: x\in\cL \textand \norm{v}\leq h(x)}
\]
is a compact subset of $\tm$. If $f:\cM\rightarrow\R$ is twice continuously differentiable, then there exists a constant $L$ such that for all $(x,v)\in\cS$, we have
\begin{equation}\label{eqn:R_Lip}
    \begin{aligned}
    \left| f(\cR_x(v)) - f(x)-\inprod{\grad f(x)}{v} \right| \leq \frac{L}{2}\norm{v}^2, \quad 
\norm{\grad \hat{f}_x(v)-\grad \hat{f}_x(0)} \leq L\norm{v},
\end{aligned}
\end{equation}
where $\hat{f}_x = f\circ\cR_x$.
\end{lem}


Our global convergence analysis centers on choosing suitable values for $\alpha_k$, $\beta_k$, and $\delta_k$ such that there is a sufficient reduction in the function value at each iteration. To achieve this, we require a positive constant $c$ such that
\begin{equation}
	f(x_k) - f(x_{k+1}) \geq c \norm{\grad f(x_k)}^2,\quad \forall k.
\end{equation}
By utilizing Proposition 4.7 in \cite{boumal2020introduction}, we can guarantee the global convergence of our RRAM method. Therefore it's necessary to derive the upper and lower bounds of $\Delta x_k^k$ (i.e., $\Delta x_k^k=x_{k+1}-x_k$) in terms of $\|\nabla f(x_k)\|$ or equivalently $\|r_k\|$. To achieve this, we need to analyze the boundedness of $\cH_k$, where $\cH_k(v) = \beta_k v-\alpha_k\cE_k(v)$, $\cE_k(v) = (X_k+\beta_k R_k)\Gamma_k(v)$, and $\Gamma_k(v)=\argmin_{\Gamma\in\R^{m_k}} \|v-R_k\Gamma\|^2_{x_k}+\delta_k\norm{X_k\Gamma}_{x_k}^2$. Notably, $\Gamma_k$ is a linear map due to it being a solution of a least-squares problem. Both $\cE$ and $\cH$ are also linear operators on $\TM{x_k}$, and we have $\dx{k}{k} = \cH_k(r_k)$. It remains to study the boundedness of $\cH_k$ and the boundedness analysis is based on prior work by Wei and Li in \cite{wei2021stochastic}.

\textbf{Lower Bound: }
We aim to find $\alpha_k$ such that the operator $\cH_k$ is coercive, namely there exists a constant $\mu\in(0,1)$, such that for all $v\in\TM{x_k}$, we have 
\begin{equation}\label{eqn:lb}
    \inprod{v}{\cH_k(v)} \geq \beta_k \mu\norm{v}^2. 
\end{equation}
Denote $\cH_k^*, \cE_k^*$ the adjoint operator of $\cH_k,\cE_k$ respectively. By the linearity of $\cH_k$, inequality \eqref{eqn:lb} is equivalent to 
\[
\lambda_{\text{min}}\left(\frac{1}{2}\left( \cH_k+\cH_k^* \right)\right) \geq \beta_k\mu .
\]
But $\lambda_{\text{min}}\left(\frac{1}{2}\left( \cH_k+\cH_k^* \right)\right) = \beta_k-\frac{1}{2}\alpha_k\lambda_{\text{max}}\left( \cE_k+\cE_k^* \right)$. Let $\lambda_k = \lambda_{\text{max}}\left( \cE_k+\cE_k^* \right)$. It suffices to choose $\alpha_k \geq 0$  satisfying 
\begin{equation}\label{eqn:alpha}
    \alpha_k \lambda_k \leq 2\beta_k(1-\mu).
\end{equation}

\textbf{Upper Bound: }
The following lemma, inspired by \cite[Lemma 1]{wei2021stochastic}, provides an upper bound estimation of $\cH_k$ and its proof can be found in the supplementary materials.
\begin{lem}\label{lem:upper_bound}
Suppose $\{x_k\}$ is generated by algorithm \ref{alg:RAM} and $\alpha_k \geq 0, \beta_k,\delta_k > 0$. Then for all $v\in\TM{x_k}$, we have:
\[
\norm{\cH_k(v)}^2 \leq 2\left[ \beta_k^2(1+2\alpha_k^2-2\alpha_k)+\alpha_k^2\delta_k^{-1} \right]\norm{v}^2.
\]
\end{lem}

\subsection{Global convergence of RRAM}
Now we arrive at the main theorem of this section, which is a direct corollary of \cite[Proposition 4.7]{boumal2020introduction}:

\begin{thm}
Suppose the \cref{asmp:global_base} and \cref{asmp:compact_levelset}  are satisfied. There exists a constant $L>0$, such that if we choose $\beta_k$ satisfying $0<\inf_k\beta_k \leq \sup_k\beta_k < \min \{1,\frac{\mu}{2L}\}$ and suppose that $\alpha_k$ satisfies inequality \eqref{eqn:alpha} and that $\alpha_k, \delta_k$ satisfy $\left(2+\frac{1}{\delta_k\beta_k^2}\right)\alpha_k^2-2\alpha_k-1<0$, then for any initial point $x_0\in\cM$,
\begin{itemize}
    \item there exists a stationary point of $f$ denoted by $x^*$, such that the sequence $\{x_k\}$ generated by \cref{alg:RRAM} has a subsequence converging to $x^*$,
    \item if \cref{asmp:global_Lip} holds further, and if for some positive integer $K$, the minimizing geodesic between $x^*$ and $x_k$ is unique for $k \geq K$, then the whole sequence $\{x_k\}$ converges to $x^*$ with converging rate $\cO(1/\sqrt{k})$.
\end{itemize}
\end{thm}

\begin{proof}
Since the level set $\cL := \set{x\in\cM: f(x) \leq f(x_0)}$ is compact, choose $h(x) = 4\norm{\grad f(x)}$ in the definition of $\cS$ in \cref{lem:R_Lip}. Then by \cref{lem:R_Lip} $\cS$ is compact and there exists a constant $L > 0$ such that for all $(x,v)\in\cS$, \eqref{eqn:R_Lip} is satisfied. The choice of $\alpha_k,\beta_k$ and $\delta_k$ implies that
\[
\beta_k^2(1+2\alpha_k^2-2\alpha_k)+\alpha_k^2\delta_k^{-1} < 2\beta_k^2 < 2.
\]
Thus $\cH_k(r_k) \leq 2[\beta_k^2(1+2\alpha_k^2-2\alpha_k)+\alpha_k^2\delta_k^{-1}]\norm{r_k} < 4\norm{r_k}$. By induction, we can prove that $(x_k,\dx{k}{k}) \in\cS$ for all $k$. Now we prove that the function value achieves sufficient decrease property. Note that
\[
\begin{aligned}
 f(x_k) - f(x_{k+1}) 
& \geq  -\inprod{\grad f(x_k)}{\dx{k}{k}} - \frac{L}{2}\norm{\dx{k}{k}}^2 \\
& \geq  \beta_k\mu\norm{\grad f(x_k)}^2 - L\left[ \beta_k^2(1+2\alpha_k^2-2\alpha_k)+\alpha_k^2\delta_k^{-1} \right]\norm{\grad f(x_k)}^2 \\
& >  (\beta_k\mu-2L\beta_k^2)\norm{\grad f(x_k)}^2 \geq c\norm{\grad f(x_k)}^2,
\end{aligned}
\]
where $c:= \inf_k \beta_k\mu-2L\beta_k^2$ and the uniform boundedness of $\beta_k$ implies the positivity of $c$. Thus
\begin{equation}\label{eqn:global_conv}
    \begin{aligned}
        f(x_0)-f_{\text{low}} &\geq f(x_0) - f(x_N) =\sum_{k=0}^{N-1} f(x_k)-f(x_{k+1}) \\
        &\geq c\sum_{k=0}^{N-1}\norm{\grad f(x_k)}^2 \geq c N \min_{1\leq k\leq N}\norm{\grad f(x_k)}^2.
    \end{aligned}
\end{equation}
In particular, we obtain 
\[
\min_{1\leq k\leq N}\norm{\grad f(x_k)} \leq \sqrt{\frac{f(x_0)-f_{\text{low}}}{c}}\frac{1}{\sqrt{N}}.
\]
On the other hand, let $N\rightarrow \infty$ in inequality \eqref{eqn:global_conv} and obtain:
\[
f(x_0) - f_{\text{low}} \geq c\sum_{k=0}^\infty \norm{\grad f(x_k)}^2.
\]
Therefore the positive series $\sum_{k=0}^\infty \norm{\grad f(x_k)}^2$ converges and as a consequence we know that $\lim_{k\rightarrow\infty}\norm{\grad f(x_k)} = 0$. The monotonicity of $f(x_k)$ and the compactness of $\cL$ imply that there exists a limit point $x^*$ of $x_k$, such that $x^*$ is a stationary point of $f$ by the continuity of $\grad f$. Finally, if \cref{asmp:global_Lip} holds, we have $\dist(x_k,x^*) \leq \frac{1}{L_1}\norm{\grad f(x_k)} \rightarrow 0$ as $k\rightarrow \infty$, which yields the convergence of the sequence $\{x_k\}$.
\end{proof}

\subsection{Local convergence of RRAM}
Notice that if we set $\alpha_k = 1$ and $\delta_k = 0$, \cref{alg:RRAM} is reduced to \cref{alg:RAM}. We now present the local convergence result of \cref{alg:RRAM}, which is almost the same as \cref{thm:RAM_final_local_conv}.

\begin{thm}\label{thm:RRAM_local_conv}
Suppose $\set{x_k}$ is generated by \cref{alg:RRAM} and \cref{asmp:basic}-\cref{asmp:bounded_least_square} are satisfied. Denote $\Bar{r}_k = r_k-\alpha_k R_k\Gamma_k$. We further assume $0\leq \alpha_k \leq 1$. If the initial point $x_0$ satisfies $\dist(x_0,x^*) < \hat{r}$ (here the definition of $\hat{r}$ is the same as \cref{thm:RAM_final_local_conv}), we then have the following local linear convergence result:
\begin{equation}\label{eqn:RRAM_local_conv}
     \|r_{k+1}\| \leq \theta_k\left[(1-\beta_k)+\kappa\beta_k\right]\|r_k\|+\sum_{i=0}^{m_k}\cO(\|r_{k-i}\|^2)
\end{equation}
where $\theta_k = \frac{\|\Bar{r}_k\|}{\|r_k\|} \leq 1$.
\end{thm} 

The proof of the local convergence of RRAM (\cref{thm:RRAM_local_conv}) is analogous to that of RAM and thus omitted for brevity. 
\section{Experiments}\label{sec:experiment}
In this section, we evaluate two algorithms, RAM and RRAM, on five problems: the max-cut problem, minimization of the Brockett cost function, Karcher mean of positive definite matrices, and low-rank matrix completion. We compare RAM and RRAM with two other algorithms, Riemannian gradient descent (RGD) and Riemannian limited memory BFGS (RLBFGS), all of these four algorithms implemented using the Matlab Manopt toolbox \cite{boumal2014manopt}. In particular, RGD and RLBFGS methods are implemented by using \texttt{steepestdescent} and \texttt{rlbfgs} in Manopt toolbox with default settings, respectively. RAM and RRAM use a fixed value of $\beta_k = 0.6$, while the maximum iteration number is set to be 1000 and termination occurs when $\|\grad f(x_k)\| < 10^{-6}$. The default setting of $m=3$ is used for RAM and RRAM, and all four algorithms share the same initial point, generated randomly. For RGD and RLBFGS, iterations are terminated if the stepsize computed in the line-search procedure is smaller than $10^{-10}$. Each experiment is repeated 10 times, and
\begin{equation}\label{eqn:setting}
    \begin{aligned}
        &\textbf{rate} = \text{the number of times an algorithm converges /10} \\
        &\textbf{grad} = \text{the geometric mean of $\norm{\grad f}$ over 10 experiments} \\
        &\textbf{t} = \text{the geometric mean of the time an algorithm spends over 10 experiments.}
    \end{aligned}        
\end{equation}
Before presenting experimental results, we provide details of our implementations for RAM and RRAM.

\subsection{Details Concerning Numerical Implementations of \texorpdfstring{\cref{alg:RAM}}{} and \texorpdfstring{\cref{alg:RRAM}}{}}

In the analysis of local convergence of Riemannian optimization algorithms, it is commonly assumed that the mapping $g(x) = \Exp_x(-\grad f(x))$ is a contraction in some neighborhood $\cU$ of the optimal solution, as stated in \cref{asmp:contraction}. One possible approach to satisfy this assumption in numerical experiments is to scale the cost function $f(x)$ by a positive constant $\lambda>0$. Specifically, instead of optimizing the original function $f(x)$, we optimize $\lambda f(x)$, where $\lambda$ is chosen to be a moderate constant. This scaling does not affect the stationary points of the function, but alters the fixed-point mapping to $g_{\lambda}(x) := \Exp_x(-\lambda\grad f(x))$. In Euclidean space, for a $C^2$ function $f$ with a positive definite Hessian matrix at a stationary point $x^*$, the fixed-point mapping can be expressed as $g_{\lambda}(x) = x-\lambda\nabla f(x)$, which satisfies $\|\mathrm{D}g_{\lambda}(x^*)\|<1$ for sufficiently small $\lambda$, indicating that $g_{\lambda}$ is a contraction near $x^*$. In the Riemannian setting, a proper scaling factor $\lambda$ may also exist such that $\|\mathrm{D}g_{\lambda}(x^*)\|<1$ at a stationary point $x^*$. Empirical studies indicate the importance of the scaling factor $\lambda$. For all the experiments conducted on matrix manifolds, we set $\lambda=1/\max(n,m)$ for both RAM and RRAM, where $n$ and $m$ are the dimensions of the matrix $X\in\cM$.

In \cref{sec:alg}, we thoroughly investigated the local convergence property of RAM algorithm. However, we noted that the global convergence of Anderson mixing remains an open problem even in Euclidean space. To address this issue, we adopt a warm-start scheme to obtain an initial point that is more likely to be close to a stationary point. Specifically, we use RGD as the warm-start method and stop the iteration when $\|\grad f(x_k)\| < 10^{-2}$ or the maximum iteration number of 100 is reached. We emphasize that since RRAM has a global convergence property, we do not employ any warm-start scheme for RRAM.

Regarding RRAM, we do not usually need to verify the positive definite condition \eqref{eqn:alpha}. Instead, we set $\alpha_k = 1$ and check whether $\dx{k}{k}$ is a descent direction. If it is not, we set $\alpha_k = 0$. As for the choice of $\delta_k$ in Algorithm \ref{alg:RRAM}, we utilize the AdaSAM method proposed in \cite{wei2021stochastic}. Specifically, we set $\delta_k = \frac{c_1 \|r_k\|}{\|\dx{k}{k}\|}$, where $c_1 = 10^{-7}$ by default.

We first present the \textbf{rate} result of all experiments in \cref{tab:rate}, before we go into the details of each experiments in the following sections.
\begin{table}[htbp]
\centering
\begin{tabular}{|lccccc|}
\hline
\multicolumn{6}{|c|}{Max Cut}                                                                                                                                                        \\ \hline
\multicolumn{1}{|c|}{(n,p)}  & \multicolumn{1}{c|}{(1000,20)} & \multicolumn{1}{c|}{(2000,40)}  & \multicolumn{1}{c|}{(8000, 80)} & \multicolumn{1}{c|}{(10000, 100)} & (25000, 150) \\ \hline
\multicolumn{1}{|l|}{RAM}    & \multicolumn{1}{c|}{10/10}     & \multicolumn{1}{c|}{10/10}      & \multicolumn{1}{c|}{10/10}      & \multicolumn{1}{c|}{10/10}        & 10/10        \\ \hline
\multicolumn{1}{|l|}{RRAM}   & \multicolumn{1}{c|}{10/10}     & \multicolumn{1}{c|}{10/10}      & \multicolumn{1}{c|}{10/10}      & \multicolumn{1}{c|}{10/10}        & 10/10        \\ \hline
\multicolumn{1}{|l|}{RLBFGS} & \multicolumn{1}{c|}{5/10}      & \multicolumn{1}{c|}{5/10}       & \multicolumn{1}{c|}{3/10}       & \multicolumn{1}{c|}{6/10}         & 5/10         \\ \hline
\multicolumn{1}{|l|}{RGD}    & \multicolumn{1}{c|}{0/10}      & \multicolumn{1}{c|}{0/10}       & \multicolumn{1}{c|}{0/10}       & \multicolumn{1}{c|}{0/10}         & 0/10         \\ \hline
\multicolumn{6}{|c|}{Brockett cost}                                                                                                                                                  \\ \hline
\multicolumn{1}{|c|}{(n,p)}  & \multicolumn{1}{c|}{(200,5)}   & \multicolumn{1}{c|}{(800,5)}    & \multicolumn{1}{c|}{(200,10)}   & \multicolumn{1}{c|}{(800, 10)}    & (1500, 10)   \\ \hline
\multicolumn{1}{|l|}{RAM}    & \multicolumn{1}{c|}{9/10}      & \multicolumn{1}{c|}{6/10}       & \multicolumn{1}{c|}{7/10}       & \multicolumn{1}{c|}{2/10}         & 1/10         \\ \hline
\multicolumn{1}{|l|}{RRAM}   & \multicolumn{1}{c|}{9/10}      & \multicolumn{1}{c|}{9/10}       & \multicolumn{1}{c|}{5/10}       & \multicolumn{1}{c|}{2/10}         & 0/10         \\ \hline
\multicolumn{1}{|l|}{RLBFGS} & \multicolumn{1}{c|}{0/10}      & \multicolumn{1}{c|}{0/10}       & \multicolumn{1}{c|}{0/10}       & \multicolumn{1}{c|}{0/10}         & 0/10         \\ \hline
\multicolumn{1}{|l|}{RGD}    & \multicolumn{1}{c|}{0/10}      & \multicolumn{1}{c|}{0/10}       & \multicolumn{1}{c|}{0/10}       & \multicolumn{1}{c|}{0/10}         & 0/10         \\ \hline
\multicolumn{6}{|c|}{Karcher Mean}                                                                                                                                                   \\ \hline
\multicolumn{1}{|c|}{(n,m)}  & \multicolumn{1}{c|}{(100,20)}  & \multicolumn{1}{c|}{(200,10)}   & \multicolumn{1}{c|}{(500, 5)}   & \multicolumn{1}{c|}{(800, 3)}     & (1000, 2)    \\ \hline
\multicolumn{1}{|l|}{RAM}    & \multicolumn{1}{c|}{10/10}     & \multicolumn{1}{c|}{10/10}      & \multicolumn{1}{c|}{10/10}      & \multicolumn{1}{c|}{10/10}        & 10/10        \\ \hline
\multicolumn{1}{|l|}{RRAM}   & \multicolumn{1}{c|}{10/10}     & \multicolumn{1}{c|}{10/10}      & \multicolumn{1}{c|}{10/10}      & \multicolumn{1}{c|}{10/10}        & 10/10        \\ \hline
\multicolumn{1}{|l|}{RLBFGS} & \multicolumn{1}{c|}{10/10}     & \multicolumn{1}{c|}{10/10}      & \multicolumn{1}{c|}{10/10}      & \multicolumn{1}{c|}{10/10}        & 9/10         \\ \hline
\multicolumn{1}{|l|}{RGD}    & \multicolumn{1}{c|}{10/10}     & \multicolumn{1}{c|}{10/10}      & \multicolumn{1}{c|}{9/10}       & \multicolumn{1}{c|}{10/10}        & 9/10         \\ \hline
\multicolumn{6}{|c|}{Matrix Completion}                                                                                                                                              \\ \hline
\multicolumn{1}{|c|}{(n,k)}  & \multicolumn{1}{c|}{(5000,20)} & \multicolumn{1}{c|}{(10000,20)} & \multicolumn{1}{c|}{(2000, 40)} & \multicolumn{1}{c|}{(5000, 40)}   & (10000, 40)  \\ \hline
\multicolumn{1}{|l|}{RAM}    & \multicolumn{1}{c|}{10/10}     & \multicolumn{1}{c|}{10/10}      & \multicolumn{1}{c|}{10/10}      & \multicolumn{1}{c|}{10/10}        & 10/10        \\ \hline
\multicolumn{1}{|l|}{RRAM}   & \multicolumn{1}{c|}{10/10}     & \multicolumn{1}{c|}{10/10}      & \multicolumn{1}{c|}{10/10}      & \multicolumn{1}{c|}{10/10}        & 10/10        \\ \hline
\multicolumn{1}{|l|}{RLBFGS} & \multicolumn{1}{c|}{10/10}     & \multicolumn{1}{c|}{10/10}      & \multicolumn{1}{c|}{10/10}      & \multicolumn{1}{c|}{10/10}        & 10/10        \\ \hline
\multicolumn{1}{|l|}{RGD}    & \multicolumn{1}{c|}{10/10}     & \multicolumn{1}{c|}{10/10}      & \multicolumn{1}{c|}{10/10}      & \multicolumn{1}{c|}{10/10}        & 10/10        \\ \hline
\end{tabular}
\caption{\textbf{rate} of each algorithms in different experiments. For the definition of \textbf{rate}, see \cref{eqn:setting}.}
\label{tab:rate}
\end{table}

\subsection{Max-cut Problem}
The max-cut problem aims to partition the vertex set of a graph into two non-empty sets such that their intersection is empty, while maximizing the weight of the edges between them \cite{10.1145/227683.227684}. A graph is represented by a weight matrix $W\in\R^{n\times n}$. However, this problem is known to be NP-hard. To address this issue, we consider the following relaxation problem \cite{article}:
\begin{equation}\label{eq:maxcut}
	\max_{V=\left[V_1, \cdots, V_n\right]} \operatorname{tr}\left(C V^{\top} V\right) \text { s.t. }\left\|V_i\right\|_2=1,\ V_i\in\mathbb{R}^p, i=1, \cdots, n,
\end{equation}
where $C = \frac{1}{4}(\operatorname{diag}(W\mathbf{1})-W)$ is the graph Laplacian matrix divided by 4, with $\mathbf{1}\in\R^n$ representing a vector of all ones. We randomly generate $W$ in our experiments and use the parameter $\textbf{tau}\in (0,1)$ to control the sparsity of the graph, with larger values of $\textbf{tau}$ leading to fewer edges. For the numerical experiments, we set the maximum iteration number to 150 and chose $m=1$ for both RAM and RRAM. The results are presented in \cref{tab:Maxcut} and \cref{fig:maxcut}, where we set $\textbf{tau}=0.3$. Our findings indicate that RAM and RRAM perform better than RLBFGS and RGD in terms of achieving higher accuracy and reaching the termination condition we set. Although the average computation time for RAM, RRAM, and RLBFGS is similar, our proposed algorithms have a superior accuracy performance on average.

\begin{figure}[htbp]
	\centering
	\includegraphics[width=0.48\textwidth]{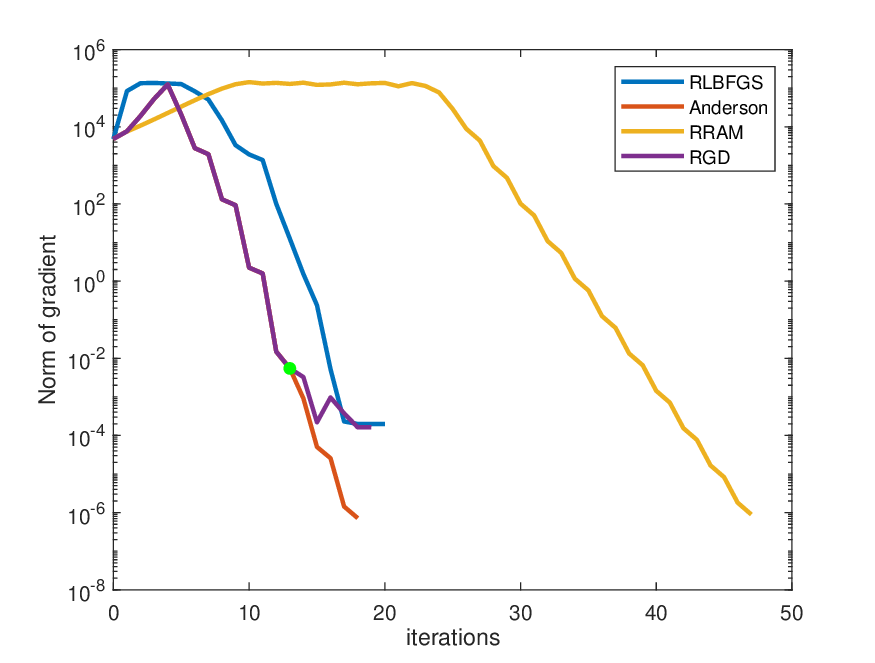}
	\includegraphics[width=0.48\textwidth]{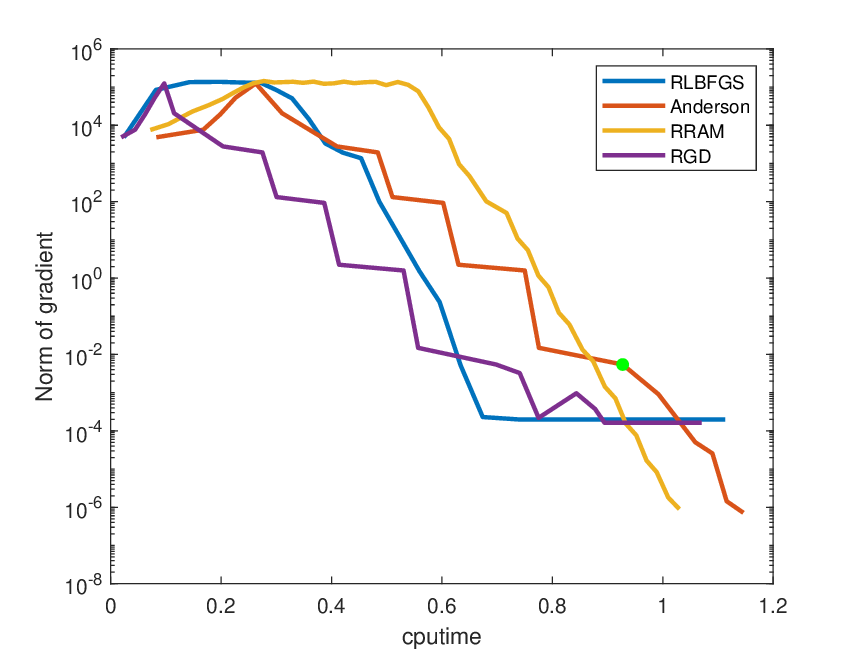}
	\caption{Max-cut problem, $n=5000, p=60$. The green point represents the moment when the warm-start of RAM ends.}
	\label{fig:maxcut}
\end{figure}

{
\begin{table}[htbp]
\centering
\begin{tabular}{|clc|c|c|c|c|c|}
\hline
\multicolumn{3}{|c|}{(n,p)}                                     & (1000,20) & (2000,40)  & (8000, 80) & (10000, 100) & (25000, 150) \\ \hline
\multicolumn{2}{|c|}{\multirow{3}{*}{RAM}} & \bgrad      & $4.06_{-7}$  & $3.94_{-7}$     & $2.65_{-7}$   & $2.29_{-7}$     & $1.02_{-7}$     \\  
\multicolumn{2}{|c|}{}                          & \Bt          & 0.12      & 0.21             & 2.41       & 4.71         & 38.85        \\  \hline
\multicolumn{2}{|c|}{\multirow{3}{*}{RRAM}}     & \bgrad      & $4.15_{-7}$  & $4.70_{-7}$     & $4.03_{-7}$   & $5.37_{-7}$     & $6.33_{-7}$     \\ 
\multicolumn{2}{|c|}{}                          & \Bt          & 0.10      & 0.23             & 2.10       & 4.47         & 38.22        \\  \hline
\multicolumn{2}{|c|}{\multirow{3}{*}{RLBFGS}}   & \bgrad      & $3.34_{-7}$  & $1.57_{-6}$     & $1.30_{-5}$   & $1.35_{-6}$     & $2.81_{-6}$     \\  
\multicolumn{2}{|c|}{}                          & \Bt         & 0.18      & 0.35            & 2.75       & 4.63         & 34.43        \\  \hline
\multicolumn{2}{|c|}{\multirow{3}{*}{RGD}}      & \bgrad      & $4.23_{-5}$  & $8.81_{-5}$     & $5.50_{-4}$   & $5.07_{-4}$     & $5.63_{-4}$     \\  
\multicolumn{2}{|c|}{}                          & \Bt         & 0.15      & 0.32             & 3.39       & 6.37         & 52.64        \\  \hline
\end{tabular}
\caption{Experiment Results: Max-cut Problem. The subscript $-k$ indicates a scale of $10^{-k}$. For the meaning of \textbf{grad} and \textbf{t}, see \cref{eqn:setting}}
\label{tab:Maxcut}
\end{table}
}

\subsection{Minimization of the Brockett Cost Function}
In this section, we study the minimization of the Brockett cost function over the Stiefel manifold. The optimization problem can be formulated as follows \cite[section 4.8.1]{absil2009optimization}:
\begin{equation*}
	\min_{X\in\mathbb{R}^{n\times p}} \operatorname{tr}(X^\top A X N), \text{ s.t. }X^\top X = I_p,
\end{equation*}
where $N = \operatorname{diag}(\mu_1,\ldots,\mu_p)$ and $\mu_1 > \cdots > \mu_p > 0$. Here, $A\in\mathbb{R}^{n\times n}$ is a symmetric matrix. This is an eigenvalue problem, as it has been proved that the columns of a global minimizer are eigenvectors corresponding exactly to the $p$ smallest eigenvalues of $A$.

In our experiments, we set $N = \operatorname{diag}(p,p-1,\ldots,1)$ and generate $A$ as $A = \frac{1}{2}(C+C^\top)$ using Matlab's $\texttt{randn}$ function. We choose $m=p+1$ and set the maximum number of iterations to be 1500. The experimental results are presented in \cref{tab:Brockett} and \cref{fig:Brockett}. As the problem size increases, all four algorithms have difficulty meeting the termination criterion. However, our RAM and RRAM methods require less computation time than RLBFGS, while RGD almost always fails to converge.

\begin{figure}[htbp]
	\centering
	\includegraphics[width=0.48\textwidth]{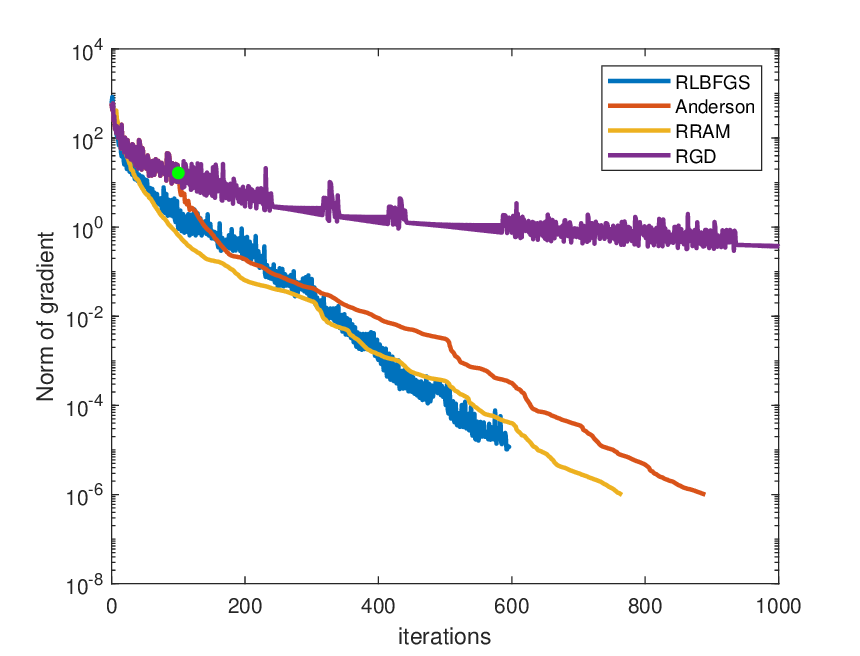}
	\includegraphics[width=0.48\textwidth]{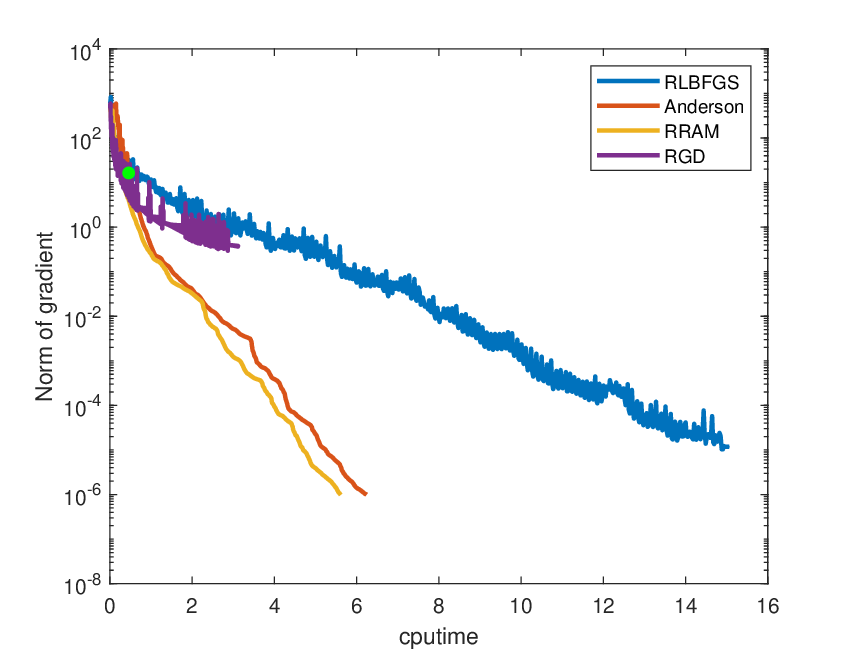}
	\caption{Minimization of Brockett cost, $n=800, p=5$. The green point represents the moment when the warm-start of RAM ends.}
	\label{fig:Brockett}
\end{figure}

\begin{table}[htbp]
\centering
\begin{tabular}{|clc|c|c|c|c|c|}
\hline
\multicolumn{3}{|c|}{(n,p)}                                        & (200,5) & (800,5)  & (200,10) & (800,10) & (1500,10) \\ \hline
\multicolumn{2}{|c|}{\multirow{3}{*}{RAM}} & \bgrad    & $1.08_{-6}$   & $1.01_{-5}$       & $1.83_{-6}$    & $2.54_{-4}$    & $4.66_{-4}$     \\ 
\multicolumn{2}{|c|}{}                          & \Bt        & 8.95       & 12.26             & 16.53       & 30.97       & 53.01        \\ 
 \hline
\multicolumn{2}{|c|}{\multirow{3}{*}{RRAM}}     & \bgrad    & $1.47_{-6}$   & $1.20_{-6}$       & $1.13_{-5}$    & $1.00_{-4}$    & $1.78_{-4}$     \\  
\multicolumn{2}{|c|}{}                          & \Bt       & 7.08       & 10.00             & 18.82       & 30.40       & 54.95        \\ 
 \hline
\multicolumn{2}{|c|}{\multirow{3}{*}{RLBFGS}}   & \bgrad   & $3.13_{-6}$   & $1.35_{-5}$       & $1.56_{-5}$    & $2.12_{-4}$    & $4.77_{-4}$     \\  
\multicolumn{2}{|c|}{}                          & \Bt        & 12.75      & 23.77            & 24.37       & 45.25       & 83.67        \\ 
 \hline
\multicolumn{2}{|c|}{\multirow{3}{*}{RGD}}      & \bgrad   & $2.86_{-1}$   & 1.50       & 1.34    & 3.08    & 4.49     \\ 
\multicolumn{2}{|c|}{}                          & \Bt       & 4.83       & 6.94               & 4.94        & 6.99        & 11.09        \\ 
 \hline
\end{tabular}
\caption{Experiment Results: Minimization of Brockett cost. The subscript $-k$ indicates a scale of $10^{-k}$. For the meaning of \textbf{grad} and \textbf{t}, see \cref{eqn:setting}}
\label{tab:Brockett}
\end{table}

\subsection{Karcher Mean of Positive Definite Matrices}
Given $m$ positive definite matrices $A_1,\ldots,A_m\in \R^{n \times n}$, the problem of finding the Karcher mean \cite{karcher1977riemannian} is to determine the optimal solution for the following optimization problem:

\begin{equation*}
	\min_{X\in\mathbb{S}^n_{++}} \sum_{k=1}^m {\rm Dist}(X, A_k),
\end{equation*}
where $\mathbb{S}^n_{++}$ denotes the set of all $n\times n$ positive definite matrices and $\operatorname{Dist}(X,Y) = \|\log(XY^{-1})\|_F$ for $X,Y\in\mathbb{S}^n_{++}$. In fact, $\operatorname{Dist}(X,Y)$ represents the Riemannian distance between $X$ and $Y$ of the manifold $\mathbb{S}^n_{++}$.

In our experiments, we generated $m$ positive definite matrices randomly and chose $A_1$ as the initial point for all algorithms. The results are presented in \cref{tab:Karcher} and \cref{fig:Karcher}. All four algorithms exhibit good convergence properties, and although RRAM is the slowest among them, our RAM algorithm outperforms all others.

\begin{figure}[htbp]
	\centering
	\includegraphics[width=0.48\textwidth]{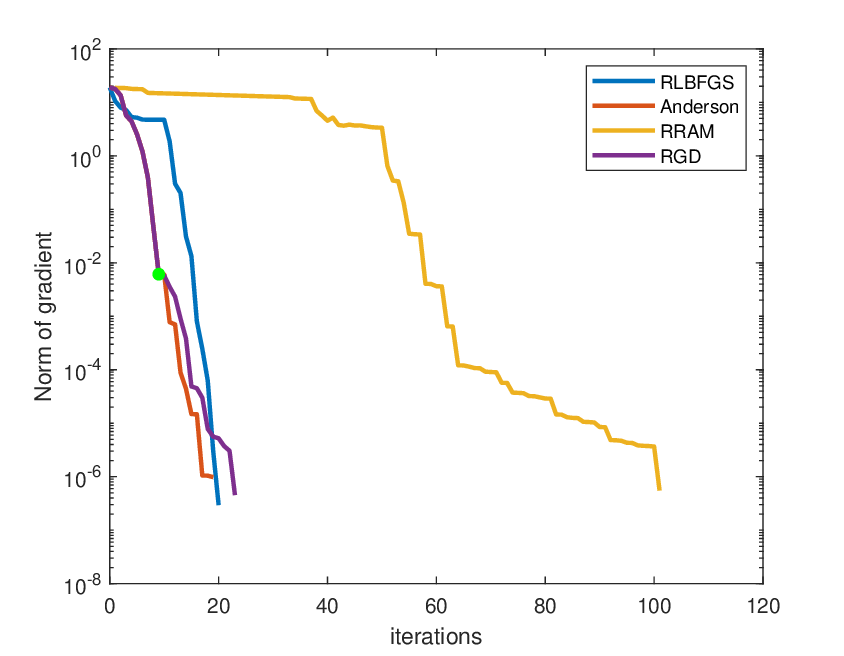}
	\includegraphics[width=0.48\textwidth]{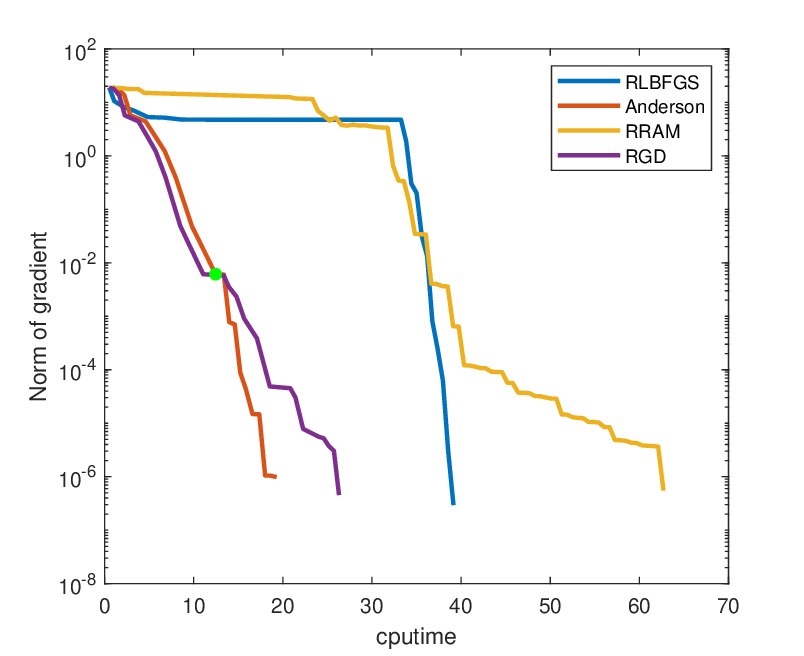}
	\caption{Karcher mean, $n=100, m=20$. The green point represents the moment when the warm-start of RAM ends.}
	\label{fig:Karcher}
\end{figure}
\begin{table}[htbp]
\centering
\begin{tabular}{|clc|c|c|c|c|c|}
\hline
\multicolumn{3}{|c|}{(n,m)}                                         & (100,20) & (200,10) & (500,5) & (800,3) & (1000,2) \\ \hline
\multicolumn{2}{|c|}{\multirow{3}{*}{RAM}} & \bgrad       & $4.65_{-7}$    & $5.87_{-7}$    & $4.93_{-7}$   & $6.64_{-7}$   & $9.25_{-8}$    \\ 
\multicolumn{2}{|c|}{}                          & \Bt              & 24.62       & 49.03       & 126.43     & 206.60     & 102.58      \\ 
 \hline
\multicolumn{2}{|c|}{\multirow{3}{*}{RRAM}}     & \bgrad      & $4.06_{-7}$    & $5.51_{-7}$    & $5.29_{-7}$   & $6.04_{-7}$   & $4.01_{-7}$    \\  
\multicolumn{2}{|c|}{}                          & \Bt             & 88.29       & 125.23      & 737.96     & 740.18     & 408.18      \\  
 \hline
\multicolumn{2}{|c|}{\multirow{3}{*}{RLBFGS}}   & \bgrad       & $1.61_{-7}$    & $3.74_{-7}$    & $4.32_{-7}$   & $3.98_{-7}$   & $1.55_{-7}$    \\ 
\multicolumn{2}{|c|}{}                          & \Bt            & 52.78       & 80.50       & 468.87     & 736.52     & 332.15      \\  
 \hline
\multicolumn{2}{|c|}{\multirow{3}{*}{RGD}}      & \bgrad       & $1.55_{-7}$    & $5.13_{-7}$    & $2.93_{-7}$   & $4.82_{-7}$   & $5.17_{-7}$    \\ 
\multicolumn{2}{|c|}{}                          & \Bt            & 31.92       & 52.73       & 138.94     & 223.42     & 171.51      \\  
 \hline
\end{tabular}
\caption{Experiment Results: Karcher Means of PSD Matrices. The subscript $-k$ indicates a scale of $10^{-k}$. For the meaning of \textbf{grad} and \textbf{t}, see \cref{eqn:setting}}
\label{tab:Karcher}
\end{table}

\subsection{Low-Rank Matrix Completion}
The original low-rank matrix completion problem seeks to find a matrix $X$ of rank $k$ that best approximates a given matrix $A$ with missing entries in $\Omega$. Directly solving the rank optimization problem is difficult, so a relaxation problem is typically used instead. One such relaxation problem is defined as follows \cite{vandereycken2013low}:
$$
\min _{X \in \mathbb{R}^{n \times m}}\left\|\mathbf{P}_{\Omega}(X-A)\right\|_F^2 \text { s.t. } \operatorname{rank}(X)=k,
$$
where $\mathbf{P}_{\Omega}$ is an orthogonal projection satisfying $\mathbf{P}_{\Omega}(X)_{ij} = X_{ij}$ if $(i,j)\in\Omega$ and 0 otherwise, and $\operatorname{rank}(A) = k$. This is an optimization problem on the fixed-rank matrix manifold, which has several Riemannian manifold structures. In this work, we employ the Riemannian structure proposed in \cite{vandereycken2013low}, which was implemented in the \texttt{fixedrankembeddedfactory} of the Manopt toolbox. To generate test data, we first randomly generated $L\in\R^{m\times k}$ and $R\in\R^{k\times n}$ from the standard normal distribution and set $A = LR$. We chose a scaling factor $\tau = \frac{3k(m+n-k)}{mn}$ to determine the set of observed entries $\Omega$, where an $m\times n$ matrix $C$ was constructed using the Matlab function \texttt{randn} and $\Omega$ was the set of all $(i,j)$ such that $C_{i j} < \tau$. For convenience, we set $m=n$ for all our experiments. Experimental results are presented in \cref{tab:Matrix_Completion} and \cref{fig:MatrixCompletion}. Our proposed algorithm, RAM, outperforms other algorithms in terms of both accuracy and computational efficiency.

\begin{figure}[htbp]
	\centering
	\includegraphics[width=0.48\textwidth]{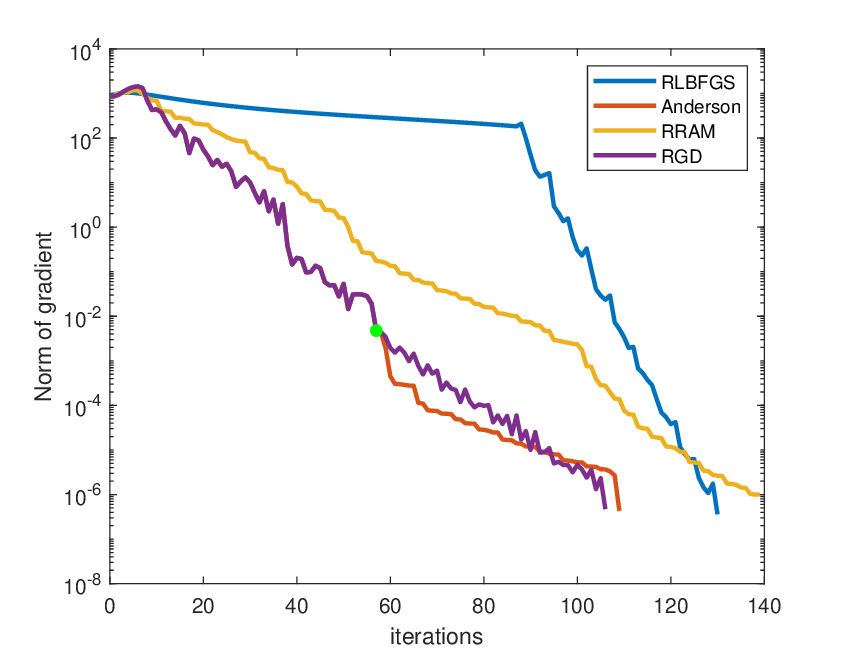}
	\includegraphics[width=0.48\textwidth]{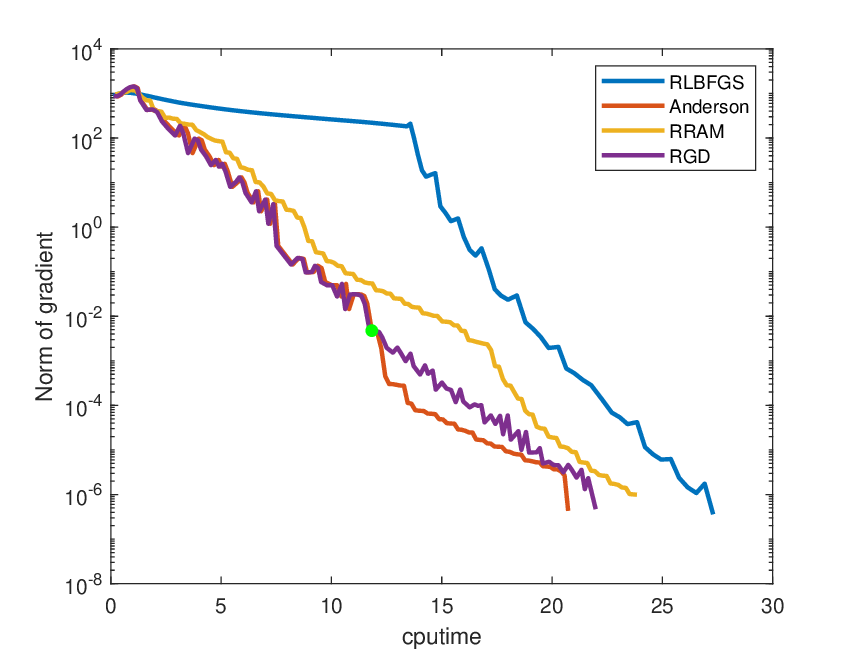}
	\caption{Matrix Completion, $n=5000, k=20$. The green point represents the moment when the warm-start of RAM ends.}
	\label{fig:MatrixCompletion}
\end{figure}
\begin{table}[htbp]
\centering
\begin{tabular}{|clc|c|c|c|c|c|}
\hline
\multicolumn{3}{|c|}{(n,k)}                                         & (5000,20) & (10000,20) & (2000,40) & (5000,40) & (10000,40) \\ \hline
\multicolumn{2}{|c|}{\multirow{3}{*}{RAM}} & \bgrad         & $8.14_{-7}$     & $8.39_{-7}$      & $8.51_{-7}$     & $6.78_{-7}$     & $7.19_{-7}$      \\ 
\multicolumn{2}{|c|}{}                          & \Bt                 & 24.51        & 73.35         & 4.42         & 23.50        & 72.14         \\  
 \hline
\multicolumn{2}{|c|}{\multirow{3}{*}{RRAM}}     & \bgrad         & $8.09_{-7}$     & $8.47_{-7}$      & $8.41_{-7}$     & $8.94_{-7}$     & $7.90_{-7}$      \\  
\multicolumn{2}{|c|}{}                          & \Bt              & 35.40        & 118.34        & 6.55         & 29.64        & 93.17         \\  
 \hline
\multicolumn{2}{|c|}{\multirow{3}{*}{RLBFGS}}   & \bgrad         & $6.43_{-7}$     & $5.52_{-7}$      & $6.28_{-7}$     & $6.90_{-7}$     & $6.93_{-7}$      \\  
\multicolumn{2}{|c|}{}                          & \Bt               & 39.07        & 261.60        & 8.54         & 33.91        & 172.85        \\ 
 \hline
\multicolumn{2}{|c|}{\multirow{3}{*}{RGD}}      & \bgrad       & $7.22_{-7}$     & $7.37_{-7}$      & $6.72_{-7}$     & $8.26_{-7}$     & $7.87_{-7}$      \\  
\multicolumn{2}{|c|}{}                          & \Bt               & 28.79        & 79.96         & 3.80         & 26.12        & 74.21         \\  
 \hline
\end{tabular}
\caption{Experiment Results: Matrix Completion. The subscript $-k$ indicates a scale of $10^{-k}$. For the meaning of \textbf{grad} and \textbf{t}, see \cref{eqn:setting}}
\label{tab:Matrix_Completion}
\end{table}

\section{Conclusion and Future Work}\label{sec:conclusion}
In this paper, we propose a Riemannian Anderson Mixing (RAM) method and its safeguarded variant, the Regularized Riemannian Anderson Mixing (RRAM) method, for solving general smooth manifold optimization problems. We establish local and global convergence analysis for RAM and RRAM, respectively. In addition, we perform numerical experiments on various manifold optimization problems, demonstrating that our algorithms outperform many classic algorithms and can be applied to a wide range of problems. 


\section*{Declarations}
\begin{itemize}
\item Funding: This work is supported by the National Key R$\&$D Program of China (No. 2021YFA1001300), National Natural
Science Foundation of China (No.12271291).
\end{itemize}

\begin{appendices}

\section{Proofs of the Lemmas and Propositions}

Here we give detailed proofs and explanations missing in the main article. 


\subsection{Manifold \texorpdfstring{$\cS^{n-1}$}{} Satisfies \texorpdfstring{\cref{asmp:retr}}{} }
Choose $\cU = \cM = \cS^{n-1}$ and set the retraction to be $\cR_x(v) = \frac{x+v}{\|x+v\|_2}$. The following proposition shows that \cref{asmp:retr} actually holds in the case of sphere. 
\begin{prop}
Given any $r>0$, there exists a constant $c$ only relevant to $r$, such that for all $x\in\cM$ and $v_x\in\TM{x}$ with $\|v_x\| < r$, we have 
\begin{equation}\label{eqn:retr}
    \dist(x,\cR_x(v_x)) \geq c\|v_x\|.
\end{equation}
\end{prop}
\begin{proof}
    Since $x\in\cS^{n-1}$ and $v_x\in\TM{x}$, we know that $x^\top x = 1$ and $x^\top v_x = 0$. Therefore we have $\|x+v_x\|_2^2 = (x+v_x)^\top(x+v_x) = x^\top x+v_x^\top v_x = 1+\|v_x\|_2^2$. So $\cR_x(v_x) = \frac{x+v_x}{\|x+v_x\|_2} = \frac{x+v_x}{\sqrt{1+\|v_x\|_2^2}}$ and $x^\top\cR_x(v_x) = \frac{x^\top(x+v_x)}{\sqrt{1+\|v_x\|_2^2}} = \frac{1}{\sqrt{1+\|v_x\|_2^2}}$. Hence 
    $$
    \dist(x,\cR_x(v_x)) = \arccos\left( \frac{1}{\sqrt{1+\|v_x\|_2^2}}\right) = \arctan(\|v_x\|_2).
    $$
    Let $f(y) = \frac{\arctan y}{y}$ when $y>0$ and $f(0) = 1$. We have thus defined a continuous strictly positive function $f$ over $[0,\infty)$. Given any $r>0$, let $c = \min_{y\in [0,r]}f(y)$ and $c>0$. Notice that $\|v_x\|_2 = \|v_x\|$. So for any $v_x$ with $\|v_x\| < r$, \cref{eqn:retr} holds.
\end{proof}

\subsection{Proof of \texorpdfstring{\cref{lem:para_vect_error}}{}}
\label{SM:proof_lem_para_vect_error}
\begin{proof}
Let $\xi_x = \Exp^{-1}_x y$ and $\eta_x = \cR^{-1}_x y$. We can therefore denote $\cP_x^{y}$ by $\cP_{\xi_x}$ and denote $\cT_x^{y}$ by $\cT_{\eta_x}$. The compact set $\cU$ can be covered by finite charts $(\varphi_1,\cU_1),\ldots,(\varphi_m,\cU_m)$. Let $L^1_{\cP}(\hat{x}, \hat{\xi})$ and $L^1_{\cT}(\hat{x},\hat{\eta})$ denote the coordinate form of $\cP_{\xi_x}$ and $\cT_{\eta_x}$ under the chart $(\varphi_1,\cU_1)$ respectively. Since all norms are equivalent in a finite dimensional space and $\cU$ is compact, there exist constants $0<b_1<b_2$ such that for all $x\in\cU, \zeta_x\in\TM{x}$, we have $b_1\|\zeta_x\| \leq \|\hat{\zeta}_x\|_2 \leq b_2\|\zeta_x\|$. 
Now for all $x\in\cU_1$, it follows that
\[
\begin{aligned}
  &\|\cP_x^y v_x-\cT_x^y v_x\| \\
  \leq& \frac{1}{b_1}\|(L^1_{\cP}(\hat{x},\hat{\xi})-L^1_{\cT}(\hat{x},\hat{\eta}))\hat{v}\|_2 \\
  \leq& \frac{1}{b_1}\|\hat{v}\|_2\left(\|(L^1_{\cP}(\hat{x},\hat{\xi})-L^1_{\cP}(\hat{x},0)\|_2+\|L^1_{\cT}(\hat{x},0)-L^1_{\cT}(\hat{x},\hat{\eta}))\|_2\right) \\
  \leq& \frac{c_1}{b_1}\|\hat{v}\|_2\left(\|\hat{\xi}\|_2+\|\hat{\eta}\|_2\right) 
  \leq \frac{c_1 b_2^2}{b_1}\|v_x\|\left(\|\xi_x\|+\|\eta_x\|\right) \\
  \leq& (1+\frac{1}{\Tilde{C}})\frac{c_1 b_2^2}{b_1}\dist(x,y)\|v_x\| 
  := \rho_1\dist(x,y)\|v_x\|,
\end{aligned}
\]
where $c_1$ is a positive constant related to $\cU_1$,  whose existence is guaranteed by the smoothness of $L^1_{\cP}$ and $L^1_{\cT}$. The second inequality holds because $L^1_{\cT}(\hat{x},0) = L^1_{\cP}(\hat{x},0)$ and \cref{asmp:retr} implies the last inequality. 
Finally, let $\rho = \max \{\rho_1,\ldots,\rho_m\}$ and we complete the proof.
\end{proof}

\subsection{Proof of \texorpdfstring{\cref{coro:multi_para_vect_err}}{} }
\label{SM:proof_coro_multi_para_vect_err}
\begin{proof}
We prove this by induction. The case when $n = 2$ is exactly  \cref{lem:para_vect_error}. Assume the proposition holds for the case of $n-1$. Consider the case of $n$. Notice that the parallel translation is isometric. Then by the induction hypothesis, we have:
\[
\begin{aligned}
   &\|\cP_{x_{n-1}}^{x_{n}}\cP_{x_{n-2}}^{x_{n-1}}\cdots\cP_{x_{1}}^{x_{2}}\ v-\cT_{x_{n-1}}^{x_{n}}\cT_{x_{n-2}}^{x_{n-1}}\cdots\cT_{x_{1}}^{x_{2}}\ v\| \\
   \leq&\|\cP_{x_{n-1}}^{x_{n}}\cP_{x_{n-2}}^{x_{n-1}}\cdots\cP_{x_{1}}^{x_{2}}\ v-\cP_{x_{n-1}}^{x_{n}}\cT_{x_{n-2}}^{x_{n-1}}\cdots\cT_{x_{1}}^{x_{2}}\ v\|+ \|\cP_{x_{n-1}}^{x_{n}}\cT_{x_{n-2}}^{x_{n-1}}\cdots\cT_{x_{1}}^{x_{2}}\ v-\cT_{x_{n-1}}^{x_{n}}\cT_{x_{n-2}}^{x_{n-1}}\cdots\cT_{x_{1}}^{x_{2}}\ v\| \\
   \leq&\|\cP_{x_{n-2}}^{x_{n-1}}\cdots\cP_{x_{1}}^{x_{2}}\ v-\cT_{x_{n-2}}^{x_{n-1}}\cdots\cT_{x_{1}}^{x_{2}}\  v\| +\rho\dist(x_{n-1},x_n)\|\cT_{x_{n-2}}^{x_{n-1}}\cdots\cT_{x_{1}}^{x_{2}}\  v\| \\
   \leq&\rho\|v\|\sum_{i=1}^{n-2}M_{\cU}^{i-1}\dist(x_i,x_{i+1})+\rho M_{\cU}^{n-2}\|v\|\dist(x_{n-1},x_{n}) 
   =\rho\|v\|\sum_{i=1}^{n-1}M_{\cU}^{i-1}\dist(x_i,x_{i+1}).
\end{aligned}
\]

\end{proof}

\subsection{Proof of \texorpdfstring{\cref{coro:multi_para}}{}}
\label{SM:proof_coro_multi_para}
\begin{proof}
We prove this by induction. The case when $n = 3$ is exactly \cref{lem:2parallel}. Assume the proposition holds for the case of $n$. Consider the case of $n+1$. Notice that parallel translation is isometric. Then by \cref{lem:2parallel} and induction hypothesis, we have:
$$
\begin{aligned}
   &\left\| \cP_{x_{1}}^{x_{n+1}}\ \xi - \cP_{x_{n}}^{x_{n+1}}\cP_{x_{n-1}}^{x_{n}}\cdots\cP_{x_{1}}^{x_{2}}\ \xi \right\| \\
   \leq& \left\| \cP_{x_{1}}^{x_{n+1}}\ \xi - \cP_{x_{n}}^{x_{n+1}}\cP_{x_{1}}^{x_{n}}\ \xi \right\|+ \left\| \cP_{x_{n}}^{x_{n+1}}\cP_{x_{1}}^{x_{n}}\ \xi - \cP_{x_{n}}^{x_{n+1}}\cP_{x_{n-1}}^{x_{n}}\cdots\cP_{x_{1}}^{x_{2}}\ \xi \right\| \\
   \leq& \rho_1\|\xi\|\dist(x_n,x_{n+1})\dist(x_1,x_n) + \left\| \cP_{x_{1}}^{x_{n}}\ \xi - \cP_{x_{n-1}}^{x_{n}}\cP_{x_{n-2}}^{x_{n-1}}\cdots\cP_{x_{1}}^{x_{2}}\ \xi \right\| \\
   \leq&\rho_1\|\xi\|\sum_{i=2}^{n}\dist(x_i,x_{i+1})\dist(x_1,x_i),
\end{aligned}
$$
which ends the proof.
\end{proof}

\subsection{Proof of \texorpdfstring{\cref{lem:para_jacobian}}{} }\label{SM:proof_para_jacobian}
\begin{proof}
    We first have:
    \[
    \begin{aligned}
        &\norm{\cP_{y}^x\cP_{z}^y H(z)\cP_{y}^z\cP_{x}^y - \cP_{z}^x H(z) \cP_{x}^z} \\
        \leq\ &\norm{\cP_{y}^x\cP_{z}^y H(z)\cP_{y}^z\cP_{x}^y - \cP_{z}^x H(z) \cP_{y}^z\cP_{x}^y} + \norm{\cP_{z}^x H(z)\cP_{y}^z\cP_{x}^y - \cP_{z}^x H(z) \cP_{x}^z}.
    \end{aligned}
    \]
    For all $v_x\in\TM{x}$, set $w_z := H(z)\cP_{y}^z\cP_{x}^y\ v_x \in\TM{z}$. Thus from \cref{lem:2parallel} we obtain:
    \[
    \begin{aligned}
        &\norm{\left( \cP_{y}^x\cP_{z}^y H(z)\cP_{y}^z\cP_{x}^y - \cP_{z}^x H(z) \cP_{y}^z\cP_{x}^y\right)v_x } \\
        \leq\ &\rho_1\dist(x,y)\dist(y,z)\norm{w_z} \\
        \leq\ &\rho_1\dist(x,y)\dist(y,z)\norm{H(z)}\norm{v_x}.
    \end{aligned}
    \]
    The definition of the linear operator norm $\norm{H(z)}$ yields:
    \[
    \norm{\cP_{y}^x\cP_{z}^y H(z)\cP_{y}^z\cP_{x}^y - \cP_{z}^x H(z) \cP_{y}^z\cP_{x}^y} \leq \rho_1\dist(x,y)\dist(y,z)\norm{H(z)}.
    \]
    On the other hand, $\forall v_x\in\TM{x}$, the following inequality holds:
    \[
    \begin{aligned}
        &\norm{\left( \cP_{z}^x H(z)\cP_{y}^z\cP_{x}^y - \cP_{z}^x H(z) \cP_{x}^z\right)\ v_x} \\
        \leq\ &\norm{H(z)}\norm{\cP_{y}^z\cP_{x}^y\ v_x- \cP_{x}^z\ v_x} \\
        \leq\ &\rho_1\dist(x,y)\dist(y,z)\norm{H(z)}\norm{v_x}.
    \end{aligned}
    \]
    As a consequence, 
    \[
    \norm{\cP_{z}^x H(z)\cP_{y}^z\cP_{x}^y - \cP_{z}^x H(z) \cP_{x}^z} \leq \rho_1\dist(x,y)\dist(y,z)\norm{H(z)}.
    \]
    Hence we have
    \[
    \norm{\cP_{y}^x\cP_{z}^y H(z)\cP_{y}^z\cP_{x}^y - \cP_{z}^x H(z) \cP_{x}^z} \leq 2\rho_1\dist(x,y)\dist(y,z)\norm{H(z)}.
    \]
\end{proof}

\subsection{Proof of \texorpdfstring{\cref{lem:err_estimate}}{} }
\label{SM:proof_lem_err_estimate}
\begin{proof}
Notice that for $1\leq j\leq k$, 
\[
\dx{k-j}{k} =\cT_{x_{k-1}}^{x_k}\cdots\cT_{x_{k-j}}^{x_{k-j+1}}\cR^{-1}_{x_{k-j}}(x_{k-j+1}).
\]
Since $\|\cR^{-1}_{x_{k-j}}(x_{k-j+1})\| < \tilr, 1 \leq j \leq k-1$, by \cref{asmp:unibound_vt}, \cref{asmp:retr} and \cref{asmp:F_lip} we have
\[
\begin{aligned}
\|\dx{k-j}{k}\| &\leq M_{\cU}^j\|\cR^{-1}_{x_{k-j}}(x_{k-j+1})\| \leq \frac{M_{\cU}^j}{\Tilde{C}}\dist(x_{k-j},x_{k-j+1}) \\
&\leq \frac{M_{\cU}^j}{L_1\Tilde{C}}\|r_{k-j+1}-\cP_{x_{k-j}}^{x_{k-j+1}}\ r_{k-j}\| \leq \frac{M_{\cU}^j}{L_1\Tilde{C}}\left(\|r_{k-j+1}\|+\|r_{k-j}\|\right)\\.
\end{aligned}
\]
Hence for $1\leq i\leq k$, the following inequality holds:
\[
\begin{aligned}
\left\|\sum_{j=1}^i\dx{k-j}{k}\right\| &\leq \sum_{j=1}^i\|\dx{k-j}{k}\| \leq \sum_{j=1}^{i} \frac{M_{\cU}^j}{L_1\Tilde{C}}\left(\|r_{k-j+1}\|+\|r_{k-j}\|\right) \\
&= \frac{M_{\cU}}{L_1\Tilde{C}}\|r_{k}\|+\sum_{j=2}^{i} \frac{M_{\cU}^{j-1}(M_{\cU} +1)}{L_1\Tilde{C}}\|r_{k-j+1}\| + \frac{M_{\cU}^i}{L_1\Tilde{C}}\|r_{k-i}\|.
\end{aligned}
\]
Set $M_1^\prime = \max\{ \frac{M_{\cU}^{m-1}(M_{\cU}+1)}{L_1\Tilde{C}},\frac{M_{\cU}(M_{\cU}+1)}{L_1\Tilde{C}} \}$. By \cref{asmp:bounded_least_square}, we have
\[
\begin{aligned}
 \|X_k\Gamma_k\|+\|r_k\| &\leq \|\Gamma_k\|_{\infty}\sum_{j=1}^{k}\|\dx{k-j}{k}\|+\|r_k\| \\
 &\leq M_{\Gamma}\sum_{j=1}^{k} \frac{M_{\cU}^j}{L_1\Tilde{C}}\left(\|r_{k-j+1}\|+\|r_{k-j}\|\right)+\|r_k\| \\
 &= \left(\frac{M_{\cU}M_{\Gamma}}{L_1\Tilde{C}}+1\right)\|r_{k}\|+\sum_{j=2}^{k} \frac{M_{\Gamma}M_{\cU}^{j-1}(M_{\cU}+1)}{L_1\Tilde{C}}\|r_{k-j+1}\| + \frac{M_{\cU}^kM_{\Gamma}}{L_1\Tilde{C}}\|r_{0}\|.
\end{aligned}
\]
Similarly, for $1\leq i\leq k$, we have
\[
\begin{aligned}
 \sum_{j=1}^i |\gamma_j^k| \|\dx{k-j}{k}\| &\leq \|\Gamma_k\|_{\infty}\sum_{j=1}^{i}\|\dx{k-j}{k}\| \\
 &\leq \frac{M_{\cU}M_{\Gamma}}{L_1\Tilde{C}}\|r_{k}\|+\sum_{j=2}^{i} \frac{M_{\Gamma}M_{\cU}^{j-1}(M_{\cU}+1)}{L_1\Tilde{C}}\|r_{k-j+1}\| + \frac{M_{\cU}^i M_{\Gamma}}{L_1\Tilde{C}}\|r_{k-i}\|.
\end{aligned}
\]
Notice that $k \leq m$. Set $M_1 = \max\set{\frac{M_{\cU}M_{\Gamma}}{L_1\Tilde{C}}+1, M_{\Gamma}M_1^\prime, M_1^\prime}$ and we complete the proof.
\end{proof}

\subsection{Proof of \texorpdfstring{\cref{coro:dist_control}}{}}\label{SM:proof_coro_dist_control}
\begin{proof}
    First of all, for any $1\leq i,j \leq k$, $\dist(x_i,x_j) \leq \dist(x_i,x^*)+\dist(x^*,x_j) < 2r^\prime < \frac{\tilr}{2}$. Also by \cref{asmp:F_lip} we have $\|r_i\| \leq L_2\dist(x_i,x^*) < L_2 r^\prime$, since $\dist(x_i,x^*) < r^\prime < 2\tilr$. So \cref{lem:err_estimate} yields
    \[
    \|\eta\| \leq \sum_{j=1}^i |\gamma_j^k| \|\dx{k-j}{k}\| + \|v\| \leq M_1\sum_{j=0}^i\|r_{k-j}\| + \|r_k\| \leq (mM_1+1)L_2 r^\prime < \frac{\tilr}{2}.
    \]
    Similarly $\|\xi\| < \frac{\tilr}{2}$. In particular, $\dist(y,x_k) = \|\eta\| < \frac{\tilr}{2}$ and $\dist(z,x_k) = \|\xi\| < \frac{\tilr}{2}$. Above all, we have $x_1,\ldots,x_{k-1},y,z\in\cB_{\cM}(x_k,\frac{\tilr}{2})$.
\end{proof}

\subsection{Proof of \texorpdfstring{\cref{prop:multi_step_err}}{} }
\label{SM:proof_prop_multi_step_err}
\begin{proof}
By \cref{coro:dist_control} we know that for any $1\leq i,j \leq k, \dist(x_i,x_j) \leq \dist(x_i,x_k)+ \dist(x_k,x_j) < \tilr$. So $\Exp_{x_i}^{-1}(x_j)$ is well-defined with norm smaller than $\tilr$. We first have
\begin{equation}\label{eqn:prop1_1}
    \begin{aligned}
     \left\| \sum_{j=1}^i\dx{k-j}{k}+\Exp_{x_k}^{-1}x_{k-i} \right\| \leq& \left\| \dx{k-1}{k}+\Exp_{x_k}^{-1}x_{k-1} \right\| 
     +\left\| \Exp_{x_k}^{-1}x_{k-i}-\Exp^{-1}_{x_k}x_{k-1}-\cP_{x_{k-1}}^{x_{k}}\Exp_{k-1}^{-1}x_{k-i} \right\| \\
     +&\left\| \sum_{j=2}^i\dx{k-j}{k}+\cP_{x_{k-1}}^{x_{k}}\Exp_{x_{k-1}}^{-1}x_{k-i} \right\|.
    \end{aligned}
\end{equation}
Notice that $\dist(x_i,x_{i+1}) < \tilr, \|\cR^{-1}_{x_{i}}(x_{i+1})\| < \tilr, 1\leq i\leq k$. \cref{lem:coslaw}, \cref{lem:para_vect_error},  \cref{lem:basic}, \cref{asmp:retr} and \cref{asmp:F_lip} yield:
\begin{equation}\label{eqn:prop1_2}
    \begin{aligned}
     &\left\| \dx{k-1}{k}+\Exp_{x_k}^{-1}x_{k-1} \right\| 
     = \left\| \cP_{x_{k-1}}^{x_{k}}\Exp_{x_{k-1}}^{-1}x_{k}-\cT_{x_{k-1}}^{x_{k}}\cR^{-1}_{x_{k-1}}x_k \right\| \\
     \leq& \left\|\Exp_{x_{k-1}}^{-1}x_{k}-\cR^{-1}_{x_{k-1}}x_k \right\|+\left\| \cP_{x_{k-1}}^{x_{k}}\cR_{x_{k-1}}^{-1}x_{k}-\cT_{x_{k-1}}^{x_{k}}\cR^{-1}_{x_{k-1}}x_k \right\| \\
     \leq& \dist\left(x_k,\Exp_{x_{k-1}}(\cR^{-1}_{x_{k-1}}x_k)\right) + \sqrt{K}\dist(x_k,x_{k-1})\|\cR^{-1}_{x_{k-1}}x_k\| +\rho\dist(x_k,x_{k-1})\|\cR^{-1}_{x_{k-1}}x_k\| \\
     \leq& \left( \frac{C}{\Tilde{C}^2}+\frac{\sqrt{K}}{\Tilde{C}}+\frac{\rho}{\Tilde{C}} \right)\dist(x_{k-1},x_k)^2 
     \leq \frac{1}{\Tilde{C}L_1}\left( \frac{C}{\Tilde{C}}+\sqrt{K}+\rho \right)(\|r_{k-1}\|+\|r_k\|)^2  \\
     \leq& \frac{2}{\Tilde{C}L_1}\left( \frac{C}{\Tilde{C}}+\sqrt{K}+\rho \right)\sum_{j=0}^i\|r_{k-j}\|^2.
    \end{aligned}
\end{equation}
Note that by \cref{coro:dist_control} $\|\Exp_{x_k}^{-1}x_{k-i}-\Exp^{-1}_{x_k}x_{k-1}\| \leq \|\Exp_{x_k}^{-1}x_{k-i}\|+\|\Exp^{-1}_{x_k}x_{k-1}\| < 
 \tilr$. According to \cref{lem:triangle_rule} and \cref{lem:coslaw}, we obtain
\begin{equation}
    \begin{aligned}
     &\left\| \Exp_{x_k}^{-1}x_{k-i}-\Exp^{-1}_{x_k}x_{k-1}-\cP_{x_{k-1}}^{x_{k}}\Exp_{x_{k-1}}^{-1}x_{k-i} \right\|
     = \left\| \cP_{x_{k}}^{x_{k-1}}(\Exp_{x_k}^{-1}x_{k-i}-\Exp^{-1}_{x_k}x_{k-1})-\Exp_{x_{k-1}}^{-1}x_{k-i} \right\| \\
     \leq& \dist\left( x_{k-i},\Exp_{x_{k-1}}\left(\cP_{x_{k}}^{x_{k-1}}(\Exp_{x_k}^{-1}x_{k-i}-\Exp^{-1}_{x_k}x_{k-1})\right) \right) +\sqrt{K}\|\Exp_{x_k}^{-1}x_{k-i}-\Exp^{-1}_{x_k}x_{k-1}\| \|\Exp_{x_{k-1}}^{-1}x_{k-i}\| \\
     \leq& c_0\min\{ \|\Exp^{-1}_{x_k}x_{k-1}\|,\|\Exp_{x_k}^{-1}x_{k-i}-\Exp^{-1}_{x_k}x_{k-1}\| \}\left( \|\Exp^{-1}_{x_k}x_{k-1}\|+ \|\Exp_{x_k}^{-1}x_{k-i}-\Exp^{-1}_{x_k}x_{k-1}\|\right)^2 \\
     &+ \sqrt{K}\|\Exp_{x_k}^{-1}x_{k-i}-\Exp^{-1}_{x_k}x_{k-1}\| \|\Exp_{x_{k-1}}^{-1}x_{k-i}\|.
    \end{aligned}
\end{equation}
On the other hand ,by \cref{asmp:F_lip}, we have 
\begin{equation}\label{eqn:prop1_3}
    \begin{aligned}
        &\left\| \Exp_{x_k}^{-1}x_{k-i}-\Exp^{-1}_{x_k}x_{k-1}-\cP_{x_{k-1}}^{x_{k}}\Exp_{x_{k-1}}^{-1}x_{k-i} \right\| \\ 
     \leq& c_0\min\{ \|\Exp^{-1}_{x_k}x_{k-1}\|,\|\Exp_{x_k}^{-1}x_{k-i}-\Exp^{-1}_{x_k}x_{k-1}\| \}\left( \|\Exp^{-1}_{x_k}x_{k-1}\|+ \|\Exp_{x_k}^{-1}x_{k-i}-\Exp^{-1}_{x_k}x_{k-1}\|\right)^2 \\
     &+ \sqrt{K}\|\Exp_{x_k}^{-1}x_{k-i}-\Exp^{-1}_{x_k}x_{k-1}\| \|\Exp_{x_{k-1}}^{-1}x_{k-i}\| \\
     \leq& \frac{c_0 \tilr}{L_1^2}(3\|r_k\|+2\|r_{k-1}\|+\|r_{k-i}\|)^2+\frac{\sqrt{K}}{L_1^2}(2\|r_k\|+\|r_{k-1}\|+\|r_{k-i}\|)(\|r_{k-1}\|+\|r_{k-i}\|) \\
     \leq& \frac{14 c_0 \tilr+3\sqrt{K}}{L_1^2}\sum_{j=0}^i\|r_{k-j}\|^2.
    \end{aligned}
\end{equation}
Combining \cref{lem:para_vect_error},  \cref{lem:err_estimate} and \cref{asmp:F_lip}, we have
\begin{equation}\label{eqn:prop1_4}
    \begin{aligned}
     &\left\| \sum_{j=2}^i\dx{k-j}{k}+\cP_{x_{k-1}}^{x_{k}}\Exp_{x_{k-1}}^{-1}x_{k-i} \right\| \\
     \leq&\left\| \cT_{x_{k-1}}^{x_{k}}\left( \sum_{j=2}^i\dx{k-j}{k-1} \right)-\cP_{x_{k-1}}^{x_{k}}\left( \sum_{j=2}^i\dx{k-j}{k-1} \right) \right\| 
     +\left\| \sum_{j=2}^i\dx{k-j}{k-1}+\Exp_{x_{k-1}}^{-1}x_{k-i} \right\| \\ 
     \leq&\rho\dist(x_{k-1},x_k)\left\| \sum_{j=2}^i\dx{k-j}{k-1} \right\|+\left\| \sum_{j=2}^i\dx{k-j}{k-1}+\Exp_{x_{k-1}}^{-1}x_{k-i} \right\| \\
     \leq& \frac{m\rho M_1}{L_1}\sum_{j=0}^i\|r_{k-j}\|^2+\left\| \sum_{j=2}^i\dx{k-j}{k-1}+\Exp_{x_{k-1}}^{-1}x_{k-i} \right\|.
    \end{aligned}
\end{equation}
Set $\Tilde{M}_2 = \frac{2}{\Tilde{C}L_1}\left( \frac{C}{\Tilde{C}}+\sqrt{K}+\rho \right)+\frac{14 c_0 \tilr+3\sqrt{K}+m\rho M_1 L_1}{L_1^2}$. Inequalities \cref{eqn:prop1_1}, \cref{eqn:prop1_2}, \cref{eqn:prop1_3} and \cref{eqn:prop1_4} yield:
\[
\left\| \sum_{j=1}^i\dx{k-j}{k}+\Exp_{x_k}^{-1}x_{k-i} \right\| \leq \Tilde{M}_2\sum_{j=0}^i\|r_{k-j}\|^2+\left\| \sum_{j=2}^i\dx{k-j}{k-1}+\Exp_{x_{k-1}}x_{k-i} \right\|.
\]
Inductively, we have
\[
\begin{aligned}
 \left\| \sum_{j=1}^i\dx{k-j}{k}+\Exp_{x_k}^{-1}x_{k-i} \right\| &\leq \Tilde{M}_2\sum_{j=0}^i\|r_{k-j}\|^2+\left\| \sum_{j=2}^i\dx{k-j}{k-1}+\Exp_{x_{k-1}}x_{k-i} \right\| \leq\cdots \\
 &\leq (i-1)\Tilde{M}_2\sum_{j=0}^i\|r_{k-j}\|^2+ \left\| \dx{k-i}{k-i+1}+\Exp_{x_{k-i+1}}^{-1}x_{k-i} \right\|.
\end{aligned}
\]
Similar to inequalities \cref{eqn:prop1_2}, we have
\[
\left\| \dx{k-i}{k-i+1}+\Exp_{x_{k-i+1}}^{-1}x_{k-i} \right\| \leq \frac{2}{\Tilde{C}L_1}\left( \frac{C}{\Tilde{C}}+\sqrt{K}+\rho \right)\sum_{j=0}^i\|r_{k-j}\|^2.
\]
Let $M_2 = m \Tilde{M}_2+\frac{2}{\Tilde{C}L_1}\left( \frac{C}{\Tilde{C}}+\sqrt{K}+\rho \right)$ and the proof is then ended.
\end{proof}

\subsection{Proof of \texorpdfstring{\cref{prop:2}}{} }
\label{SM:proof_prop_2}
\begin{proof}
Choose $\alpha = 1$ and $\beta = 0$ in \cref{coro:dist_control} and we obtain that $y_k^i \in\cB_{\cM}(x_k,\frac{\tilr}{2})$. So we further have $\dist(y_k^i,x_{k-i}) < \tilr$. Notice that
\begin{equation}\label{eqn:prop2_1}
    \begin{aligned}
     &\left\| \cP_{y_k^i}^{x_k}F(y_k^i)-\cT_{x_{k-1}}^{x_k}\cT_{x_{k-2}}^{x_{k-1}}\cdots\cT_{x_{k-i}}^{x_{k-i+1}}F(x_{k-i}) \right\| \\
     \leq& \left\| \cP_{y_k^i}^{x_k}F(y_k^i)-\cP_{x_{k-i}}^{x_{k}}F(x_{k-i}) \right\|
     +\left\| \cP_{x_{k-i}}^{x_{k}}F(x_{k-i})-\cP_{x_{k-1}}^{x_{k}}\cdots\cP_{x_{k-i}}^{x_{k-i+1}}F(x_{k-i}) \right\| \\
     +&\left\| \cP_{x_{k-1}}^{x_{k}}\cdots\cP_{x_{k-i}}^{x_{k-i+1}}F(x_{k-i})-\cT_{x_{k-1}}^{x_k}\cdots\cT_{x_{k-i}}^{x_{k-i+1}}F(x_{k-i}) \right\|
    \end{aligned}
\end{equation}
Firstly, we have 
\begin{equation}
    \begin{aligned}
     & \left\| \cP_{y_k^i}^{x_k}F(y_k^i)-\cP_{x_{k-i}}^{x_{k}}F(x_{k-i}) \right\| \\
     \leq& \left\| \cP_{y_k^i}^{x_k}F(y_k^i)-\cP_{y_k^i}^{x_k}\cP_{x_{k-i}}^{y_k^i}F(x_{k-i}) \right\| +\left\| \cP_{y_k^i}^{x_k}\cP_{x_{k-i}}^{y_k^i}F(x_{k-i})-\cP_{x_{k-i}}^{x_{k}}F(x_{k-i}) \right\| \\
     \leq& L_2\dist(x_{k-i},y_k^i)+\rho_1\|F(x_{k-i})\|\dist(x_k,y_k^i). 
    \end{aligned}
\end{equation}
But by \cref{lem:err_estimate}
\[
\|F(x_{k-i})\|\dist(x_k,y_k^i)=\|r_{k-i}\| \left\|\sum_{j=1}^i\dx{k-j}{k} \right\|\leq \frac{(m+1)M_1}{2}\sum_{j=0}^i\|r_{k-j}\|^2.
\]
Since $\dist(x_k,x_{k-i}), \dist(x_k,y_k^i) < \tilr/2$, according to \cref{prop:multi_step_err}, \cref{coro:dist_control} and \cref{lem:coslaw}, we have
\[
\begin{aligned}
 \dist(x_{k-i},y_k^i) \leq& \left\| -\sum_{j=1}^i\dx{k-j}{k}-\Exp_{x_k}^{-1}x_{k-i} \right\|+\sqrt{K}\left\| \sum_{j=1}^i\dx{k-j}{k} \right\|\dist(x_k,x_{k-i}) \\
 \leq& M_2\sum_{j=0}^i\|r_{k-i}\|^2+\frac{\sqrt{K}M_1}{L_1}(\|r_k\|+\|r_{k-i}\|)\sum_{j=0}^i\|r_{k-j}\| \\
 \leq& (M_2+\frac{m\sqrt{K}M_1}{L_1})\sum_{j=0}^i \|r_{k-j}\|^2.
\end{aligned}
\]
Consequently,
\begin{equation}\label{eqn:prop2_2}
   \left\| \cP_{y_k^i}^{x_k}F(y_k^i)-\cP_{x_{k-i}}^{x_{k}}F(x_{k-i}) \right\| \leq \left[ \frac{(m+1)\rho_1 M_1}{2} + L_2(M_2+\frac{m\sqrt{K}M_1}{L_1}) \right] \sum_{j=0}^i\|r_{k-j}\|^2. 
\end{equation}
Next, \cref{coro:multi_para} and \cref{asmp:F_lip} yields:
\begin{equation}\label{eqn:prop2_3}
    \begin{aligned}
     &\left\| \cP_{x_{k-i}}^{x_{k}}F(x_{k-i})-\cP_{x_{k-1}}^{x_{k}}\cdots\cP_{x_{k-i}}^{x_{k-i+1}}F(x_{k-i}) \right\| 
     \leq \rho_1\|F(x_{k-i})\|\sum_{j=1}^{i-1}\dist(x_{k-j},x_{k-j+1}) \\
     \leq& \sum_{j=1}^{i-1}\frac{\rho_1}{L_1}\|r_{k-i}\|(\|r_{k-j}\|+\|r_{k-j+1}\|) 
     \leq \frac{2\rho_1}{L_1}\sum_{j=0}^{i-1}\|r_{k-i}\|\|r_{k-j}\|\\
     \leq& \frac{2m\rho_1}{L_1}\sum_{j=0}^i\|r_{k-j}\|^2.
    \end{aligned}
\end{equation}
By \cref{coro:multi_para_vect_err} and \cref{asmp:F_lip}, we have
\begin{equation}\label{eqn:prop2_4}
    \begin{aligned}
     &\left\| \cP_{x_{k-1}}^{x_{k}}\cdots\cP_{x_{k-i}}^{x_{k-i+1}}F(x_{k-i})-\cT_{x_{k-1}}^{x_k}\cdots\cT_{x_{k-i}}^{x_{k-i+1}}F(x_{k-i}) \right\| \\
     \leq& \rho\|r_{k-i}\|\sum_{j=1}^i M_{\cU}^{i-j}\dist(x_{k-j},x_{k-j+1}) \\
     \leq& \frac{\rho}{L_1}\max\{M_\cU,M_\cU^m\}\|r_{k-i}\|\sum_{j=1}^i (\|r_{k-j}\|+\|r_{k-j+1}\|) \\
     \leq& \frac{2m\rho}{L_1}\max\{M_\cU,M_\cU^m\}\sum_{j=0}^i \|r_{k-j}\|^2.
    \end{aligned}
\end{equation}
Set $M_3 = \left[ \frac{(m+1)\rho_1 M_1}{2} + L_2(M_2+\frac{m\sqrt{K}M_1}{L_1}) \right]+\frac{2m\rho_1}{L_1}+\frac{2m\rho}{L_1}\max\{M_\cU,M_\cU^m\}$. Combining inequalities \cref{eqn:prop2_1}, \cref{eqn:prop2_2}, \cref{eqn:prop2_3} and \cref{eqn:prop2_4}, we complete the proof.

\end{proof}

\subsection{Proof of \texorpdfstring{\cref{prop:3}}{} }\label{SM:proof_prop_3}
\begin{proof}
    To simplify the notations, fix $1\leq i \leq k$ and set $x = x_k, v_1 = -w_k^{i-1}, v_2 = -z_k^{i-1}, \Bar{v} = -\dx{k-i}{k}, \lambda = \gamma_i^k, y_1 = \Exp_x(v_1) = v_k^{i-1}, y_2 = \Exp_x(v_2) = y_k^{i-1}, z_1 = \Exp_x(v_1+\lambda \Bar{v}) = v_k^i, z_2 = \Exp_x(v_2+\Bar{v}) = y_k^i$. Define $\Tilde{z}_1 = \Exp_{y_1}(\lambda \cP_{x}^{y_1}\ \Bar{v}), \Tilde{z}_2 = \Exp_{y_1}( \cP_{x}^{y_2}\ \Bar{v})$. Since $x_j\in\cB_{\cM}(x^*,r^\prime), 1\leq j\leq k$, by\cref{coro:dist_control} we know that $\dist(x_{k-l},x_{k-j}) <  \frac{\tilr}{2}$ for $1\leq l,j \leq k$ and that $y_1,y_2,z_1,z_2, \Tilde{z}_1,\Tilde{z}_2 \in \cB_{\cM}(x,\frac{\tilr}{2})$. So the distance of any two points among $x,y_1,y_2,z_1,z_2,\Tilde{z}_1,\Tilde{z}_2$ is no larger than $\tilr$. Then we have:
    \begin{equation}\label{eqn:prop_3_1}
        \begin{aligned}
        &\norm{\cP_{z_1}^x F(z_1)-\cP_{y_1}^x F(y_1)-\lambda (\cP_{z_2}^x F(z_2)-\cP_{y_2}^x F(y_2))} \\
        \leq\ &\norm{\cP_{z_1}^x F(z_1)- \cP_{\Tilde{z}_1}^x F(\Tilde{z}_1)}+\norm{\cP_{z_2}^x F(z_2)- \cP_{\Tilde{z}_2}^x F(\Tilde{z}_2)} 
        +\norm{\cP_{\Tilde{z}_1}^x F(\Tilde{z}_1)-\cP_{y_1}^x F(y_1)-\lambda (\cP_{\Tilde{z}_2}^x F(\Tilde{z}_2)-\cP_{y_2}^x F(y_2))}.
    \end{aligned}
    \end{equation}
    The Lipschitz continuity of $F$ , \cref{lem:triangle_rule} and \cref{lem:2parallel} yield:
    \begin{equation}\label{eqn:prop_3_2}
        \begin{aligned}
        &\norm{\cP_{z_1}^x F(z_1)- \cP_{\Tilde{z}_1}^x F(\Tilde{z}_1)} \\
        \leq\ &\norm{F(z_1)- \cP_{\Tilde{z}_1}^{z_1} F(\Tilde{z}_1) } + \norm{\cP_{\Tilde{z}_1}^{z_1} F(\Tilde{z}_1) - \cP_{x}^{z_1}\cP_{\Tilde{z}_1}^{x} F(\Tilde{z}_1) } \\
        \leq\ &L_2\dist(z_1,\Tilde{z}_1)+\rho_1\dist(x,\Tilde{z}_1)\dist(x,z_1)\norm{F(\Tilde{z}_1)} \\
        \leq\ &L_2 c_0\min\{\norm{v_1},|\lambda|\norm{\Bar{v}}\}(\norm{v_1}+|\lambda|\norm{\Bar{v}})^2      +\rho_1(\norm{v_1}+|\lambda|\norm{\Bar{v}})^2\left( \norm{F(x)}+L_2(\norm{v_1}+|\lambda|\norm{\Bar{v}}) \right) \\
        \leq\ & \frac{m\tilr M_1^2}{2}\left[ L_2c_0+\rho_1(1+L_2) \right]\sum_{j=0}^i\|r_{k-j}\|^2.
    \end{aligned}
    \end{equation}
    Similarly, we have
    \begin{equation}\label{eqn:prop_3_3}
        \norm{\cP_{z_2}^x F(z_2)- \cP_{\Tilde{z}_2}^x F(\Tilde{z}_2)} \leq \frac{m\tilr M_1^2}{2}\left[ L_2c_0+\rho_1(1+L_2) \right]\sum_{j=0}^i\|r_{k-j}\|^2.
    \end{equation}
    Note that
    \begin{equation}\label{eqn:prop_3_4}
        \begin{aligned}
        &\norm{\cP_{\Tilde{z}_1}^x F(\Tilde{z}_1)-\cP_{y_1}^x F(y_1)-\lambda (\cP_{\Tilde{z}_2}^x F(\Tilde{z}_2)-\cP_{y_2}^x F(y_2))} \\
        \leq\ &\norm{\cP_{\Tilde{z}_1}^{x} F(\Tilde{z}_1) - \cP_{y_1}^{x}\cP_{\Tilde{z}_1}^{y_1} F(\Tilde{z}_1) } + \norm{\cP_{\Tilde{z}_2}^{x} F(\Tilde{z}_2) - \cP_{y_2}^{x}\cP_{\Tilde{z}_2}^{y_2} F(\Tilde{z}_2) } \\+&\norm{\cP_{y_1}^x\left(\cP_{\Tilde{z}_1}^{y_1} F(\Tilde{z}_1)- F(y_1)\right)-\lambda \left(\cP_{y_2}^x\left(\cP_{\Tilde{z}_2}^{y_2} F(\Tilde{z}_2)- F(y_2)\right) \right)}. \\
    \end{aligned}
    \end{equation}
   Notice that by \cref{asmp:F_lip} $\|F(\Tilde{z}_1)\| \leq L_2\dist(x,\Tilde{z}_1)+\|F(x)\| < (1+L_2)\frac{\tilr}{2}$. Similarly $\|F(\Tilde{z}_2)\| < (1+L_2)\frac{\tilr}{2}$. By \cref{lem:2parallel}, we have
   \begin{equation}\label{eqn:prop_3_5}
       \begin{aligned}
        \norm{\cP_{\Tilde{z}_1}^{x} F(\Tilde{z}_1) - \cP_{y_1}^{x}\cP_{\Tilde{z}_1}^{y_1} F(\Tilde{z}_1) } 
    \leq\ &\rho_1\dist(x,y_1)\dist(y_1,\Tilde{z}_1)\norm{F(\Tilde{z}_1)} 
    = \rho_1\norm{v_1}\norm{\lambda\Bar{v}}\norm{F(\Tilde{z}_1)} \\
    \leq\ &\frac{\rho_1 \tilr(1+L_2)}{8}(\|v_1\|+\|\lambda\Bar{v}\|)^2 
    \leq\ \frac{\rho_1 \tilr(1+L_2)mM_1^2}{8}\sum_{j=0}^i\|r_{k-j}\|^2.
    \end{aligned}
   \end{equation}
    Similarly,
    \begin{equation}\label{eqn:prop_3_6}
        \norm{\cP_{\Tilde{z}_2}^{x} F(\Tilde{z}_2) - \cP_{y_2}^{x}\cP_{\Tilde{z}_2}^{y_2} F(\Tilde{z}_2) } \leq \frac{\rho_1 \tilr(1+L_2)mM_1^2}{8}\sum_{j=0}^i\|r_{k-j}\|^2.
    \end{equation}
    So it remains to estimate the upper bound of 
    \[
    \norm{\cP_{y_1}^x\left(\cP_{\Tilde{z}_1}^{y_1} F(\Tilde{z}_1)- F(y_1)\right)-\lambda \left(\cP_{y_2}^x\left(\cP_{\Tilde{z}_2}^{y_2} F(\Tilde{z}_2)- F(y_2)\right) \right)}.
    \]
    Define $\gamma_1(t) := \Exp_{y_1}(t\lambda\cP_{x}^{y_1}\ \Bar{v})$ and $\gamma_2(t) := \Exp_{y_2}(t\cP_{x}^{y_2}\ \Bar{v})$. Notice that $\dist(x^*,\gamma_1(t)) \leq \dist(x^*,x)+\dist(x,y_1)+\dist(y_1,\gamma(t)) = \dist(x^*,x)+\|v_1\|+t|\gamma|\|\Bar{v}\| < r^\prime+\frac{\tilr}{2} < \tilr$ and similarly $\dist(x^*,\gamma_1(t)) < \tilr$. So $\dist(\gamma_1(t),\gamma-2(t)) < 2\tilr$. Set $\mu = \max\{\|H(p)\|: \dist(p,x^*)\leq \tilr\}$. Taylor formula yields
    \[
    \begin{aligned}
        \cP_{\Tilde{z}_1}^{y_1} F(\Tilde{z}_1)- F(y_1)
        =\  &\int_0^1\cP_{\gamma_1(t)}^{y_1}\nabla_{\dot{\gamma}_1(t)}F(\gamma_1(t)) \rmd t         =\int_0^1\cP_{\gamma_1(t)}^{y_1}H(\gamma_1(t))[\dot{\gamma}_1(t)] \rmd t \\
        =\ &\lambda\int_0^1\left(\cP_{\gamma_1(t)}^{y_1}H(\gamma_1(t))\cP_{y_1}^{\gamma_1(t)}\right)\cP_{x}^{y_1}[\Bar{v}] \rmd t.\\
    \end{aligned}
    \]
    Thus,
    \[
    \cP_{y_1}^x\left(\cP_{\Tilde{z}_1}^{y_1} F(\Tilde{z}_1)- F(y_1)\right) = \lambda\int_0^1\left(\cP_{y_1}^x\cP_{\gamma_1(t)}^{y_1}H(\gamma_1(t))\cP_{y_1}^{\gamma_1(t)}\cP_{x}^{y_1}\right)[\Bar{v}] \rmd t.
    \]
    Similarly,
    \[
    \cP_{y_2}^x\left(\cP_{\Tilde{z}_2}^{y_2} F(\Tilde{z}_2)- F(y_2)\right) = \int_0^1\left(\cP_{y_2}^x\cP_{\gamma_2(t)}^{y_2}H(\gamma_2(t))\cP_{y_2}^{\gamma_2(t)}\cP_{x}^{y_2}\right)[\Bar{v}] \rmd t.
    \]
    Hence we have
    \[
    \begin{aligned}
        &\cP_{y_1}^x\left(\cP_{\Tilde{z}_1}^{y_1} F(\Tilde{z}_1)- F(y_1)\right)-\lambda \left(\cP_{y_2}^x\left(\cP_{\Tilde{z}_2}^{y_2} F(\Tilde{z}_2)- F(y_2)\right) \right) \\
        =\ &\lambda \int_0^1\left( \cP_{y_1}^x\cP_{\gamma_1(t)}^{y_1}H(\gamma_1(t))\cP_{y_1}^{\gamma_1(t)}\cP_{x}^{y_1}-\cP_{y_2}^x\cP_{\gamma_2(t)}^{y_2}H(\gamma_2(t))\cP_{y_2}^{\gamma_2(t)}\cP_{x}^{y_2}\right)[\Bar{v}] \rmd t.
    \end{aligned}
    \]
    Notice that for $0\leq t\leq 1$,
    \[
    \begin{aligned}
        &\norm{\cP_{y_1}^x\cP_{\gamma_1(t)}^{y_1}H(\gamma_1(t))\cP_{y_1}^{\gamma_1(t)}\cP_{x}^{y_1}-\cP_{y_2}^x\cP_{\gamma_2(t)}^{y_2}H(\gamma_2(t))\cP_{y_2}^{\gamma_2(t)}\cP_{x}^{y_2}} \\
        \leq\ &\norm{\cP_{y_1}^x\cP_{\gamma_1(t)}^{y_1}H(\gamma_1(t))\cP_{y_1}^{\gamma_1(t)}\cP_{x}^{y_1}-\cP_{\gamma_1(t)}^{x}H(\gamma_1(t))\cP_{x}^{\gamma_1(t)}}+\norm{\cP_{\gamma_1(t)}^{x}H(\gamma_1(t))\cP_{x}^{\gamma_1(t)}-\cP_{\gamma_2(t)}^{x}H(\gamma_2(t))\cP_{x}^{\gamma_2(t)}} \\
        +&\norm{\cP_{\gamma_2(t)}^{x}H(\gamma_2(t))\cP_{x}^{\gamma_2(t)} - \cP_{y_2}^x\cP_{\gamma_2(t)}^{y_2}H(\gamma_2(t))\cP_{y_2}^{\gamma_2(t)}\cP_{x}^{y_2}} \\
        \leq\ &\norm{\cP_{y_1}^x\cP_{\gamma_1(t)}^{y_1}H(\gamma_1(t))\cP_{y_1}^{\gamma_1(t)}\cP_{x}^{y_1}-\cP_{\gamma_1(t)}^{x}H(\gamma_1(t))\cP_{x}^{\gamma_1(t)}} +\norm{\cP_{\gamma_2(t)}^{x}H(\gamma_2(t))\cP_{x}^{\gamma_2(t)} - \cP_{y_2}^x\cP_{\gamma_2(t)}^{y_2}H(\gamma_2(t))\cP_{y_2}^{\gamma_2(t)}\cP_{x}^{y_2}} \\
        +&\norm{\cP_{\gamma_2(t)}^{\gamma_1(t)}H(\gamma_2(t))\cP_{\gamma_1(t)}^{\gamma_2(t)} - \cP_{x}^{\gamma_1(t)}\cP_{\gamma_2(t)}^{x}H(\gamma_2(t))\cP_{x}^{\gamma_2(t)}\cP_{\gamma_1(t)}^{x}} 
        +\norm{ H(\gamma_1(t)) -\cP_{\gamma_2(t)}^{\gamma_1(t)}H(\gamma_2(t))\cP_{\gamma_1(t)}^{\gamma_2(t)}}.
    \end{aligned}
    \]
    Since $0\leq t \leq 1$, \cref{lem:para_jacobian} implies that for all $0\leq t \leq 1$,
    \[
    \begin{aligned}             \norm{\cP_{y_1}^x\cP_{\gamma_1(t)}^{y_1}H(\gamma_1(t))\cP_{y_1}^{\gamma_1(t)}\cP_{x}^{y_1}-\cP_{\gamma_1(t)}^{x}H(\gamma_1(t))\cP_{x}^{\gamma_1(t)}} 
    \leq\ &2\rho_1\dist(x,y_1)\dist(y_1,\gamma_1(t))\|H(\gamma_1(t))\| \\
    \leq\ &2\rho_1\mu t\|v_1\|\|\lambda \Bar{v}\| 
    \leq\ \frac{\rho_1\mu}{2}(\|v_1\|+\|\lambda \Bar{v}\|)^2 \\
    \leq\ &\frac{\rho_1\mu mM_1^2}{2}\sum_{j=0}^{i}\|r_{k-j}\|^2.
    \end{aligned}
    \]
    Similarly, we have
    \[          \norm{\cP_{\gamma_2(t)}^{x}H(\gamma_2(t))\cP_{x}^{\gamma_2(t)} - \cP_{y_2}^x\cP_{\gamma_2(t)}^{y_2}H(\gamma_2(t))\cP_{y_2}^{\gamma_2(t)}\cP_{x}^{y_2}} \leq\ \frac{\rho_1\mu mM_1^2}{2}\sum_{j=0}^{i}\|r_{k-j}\|^2,
    \]
    and
    \[ \norm{\cP_{\gamma_2(t)}^{\gamma_1(t)}H(\gamma_2(t))\cP_{\gamma_1(t)}^{\gamma_2(t)} - \cP_{x}^{\gamma_1(t)}\cP_{\gamma_2(t)}^{x}H(\gamma_2(t))\cP_{x}^{\gamma_2(t)}\cP_{\gamma_1(t)}^{x}}  \leq\ 2\rho_1\mu mM_1^2\sum_{j=0}^{i}\|r_{k-j}\|^2,
    \]  
    And \cref{asmp:jacobian} shows that
    \[
    \begin{aligned}
        \norm{ H(\gamma_1(t)) -\cP_{\gamma_2(t)}^{\gamma_1(t)}H(\gamma_2(t))\cP_{\gamma_1(t)}^{\gamma_2(t)}}
        \leq\ &L\left( \dist(\gamma_1(t),y_1)+\dist(y_1,x)+\dist(x,y_2)+\dist(y_2,\gamma_2(t)) \right) \\
        =\ &L\left(\norm{v_1}+ \norm{v_2}+ (|\lambda|+1)\norm{\Bar{v}}\right) 
        \leq\ 2M_1L\sum_{j=0}^{i}\|r_{k-j}\|
    \end{aligned}
    \]
    holds for all $0\leq t \leq 1$. As a result, $\forall t\in [0,1]$, 
    \[
    \begin{aligned}
        &\norm{\cP_{y_1}^x\cP_{\gamma_1(t)}^{y_1}H(\gamma_1(t))\cP_{y_1}^{\gamma_1(t)}\cP_{x}^{y_1}-\cP_{y_2}^x\cP_{\gamma_2(t)}^{y_2}H(\gamma_2(t))\cP_{y_2}^{\gamma_2(t)}\cP_{x}^{y_2}} 
        \leq\ 2M_1L\sum_{j=0}^{i}\|r_{k-j}\|+3\rho_1\mu mM_1^2\sum_{j=0}^{i}\|r_{k-j}\|^2.
    \end{aligned}
    \]
    Finally, we have
    \begin{equation}\label{eqn:prop_3_7}
        \begin{aligned}
        &\norm{\cP_{y_1}^x\left(\cP_{\Tilde{z}_1}^{y_1} F(\Tilde{z}_1)- F(y_1)\right)-\lambda \left(\cP_{y_2}^x\left(\cP_{\Tilde{z}_2}^{y_2} F(\Tilde{z}_2)- F(y_2)\right) \right)} \\
        \leq\ &|\lambda|\int_0^1 \norm{\cP_{y_1}^x\cP_{\gamma_1(t)}^{y_1}H(\gamma_1(t))\cP_{y_1}^{\gamma_1(t)}\cP_{x}^{y_1}-\cP_{y_2}^x\cP_{\gamma_2(t)}^{y_2}H(\gamma_2(t))\cP_{y_2}^{\gamma_2(t)}\cP_{x}^{y_2}}\norm{\Bar{v}} \rmd t \\
        \leq\ &|\lambda|\norm{\Bar{v}}\left[2M_1L\sum_{j=0}^{i}\|r_{k-j}\|+3\rho_1\mu mM_1^2\sum_{j=0}^{i}\|r_{k-j}\|^2 \right] \\
        \leq\ & 2M_1^2L(\sum_{j=0}^{i}\|r_{k-j}\|)^2+\frac{3}{2}\rho_1\mu mM_1^2\tilr\sum_{j=0}^{i}\|r_{k-j}\|^2 \\
        \leq\ & (2L+\frac{3}{2}\rho_1\mu \tilr)mM_1^2\sum_{j=0}^{i}\|r_{k-j}\|^2.
    \end{aligned}
    \end{equation}
    Set $M_4 = m\tilr M_1^2\left[ L_2c_0+\rho_1(1+L_2) \right]+\frac{1}{4}\rho_1 \tilr(1+L_2)mM_1^2+(2L+\frac{3}{2}\rho_1\mu \tilr)mM_1^2$ and the proof is then ended by inequalities \cref{eqn:prop_3_1},\cref{eqn:prop_3_2}, \cref{eqn:prop_3_3}, \cref{eqn:prop_3_4}, \cref{eqn:prop_3_5}, \cref{eqn:prop_3_6} and \cref{eqn:prop_3_7} .
\end{proof}

\subsection{Proof of \texorpdfstring{\cref{lem:upper_bound}}{} } 
\label{SM:proof_lem_upper_bound}
\begin{proof}
The definition of $\Gamma_k(v)$ yields:
\[
\norm{v-R_k\Gamma_k (v)}^2+\delta_k\norm{X_k\Gamma_k(v)}^2 \leq \norm{v}^2.
\]
Hence,
\[
\begin{aligned}
\norm{\cH_k(v)}^2 &= \norm{\beta_k v-\alpha_k(X_k+\beta_k\R_k)\Gamma_k(v)}^2 \\
&= \norm{\beta_k(1-\alpha_k)v+\beta_k\alpha_k\left(v-R_k\Gamma_k(v)\right)-\alpha_k\delta_k^{-\frac{1}{2}}\delta_k^{\frac{1}{2}}X_k\Gamma_k(v) }^2 \\
&\leq \left[ \beta_k^2(1-\alpha_k)^2+\beta_k^2\alpha_k^2+\alpha_k^2\delta_k^{-1} \right]\left[ \norm{v}^2+\norm{v-R_k\Gamma_k (v)}^2+\delta_k\norm{X_k\Gamma_k(v)}^2 \right]\\
&\leq 2\left[ \beta_k^2(1+2\alpha_k^2-2\alpha_k)+\alpha_k^2\delta_k^{-1} \right]\norm{v}^2.
\end{aligned}
\]
\end{proof}




\end{appendices}


\bibliography{sn-article}

\end{document}